\documentclass[12pt]{article}

\usepackage{amsmath,amssymb,amsthm}

\usepackage{graphicx,color}

\date{}

\normalsize

\newtheorem{theorem}{Theorem}[section]
\newtheorem{lemma}[theorem]{Lemma}

\newtheorem{corollary}[theorem]{Corollary}
\numberwithin{equation}{section}

\def\({\bigl(}   \def\){\bigr)}
\def\abs#1{\mathopen|#1\mathclose|} \let\card=\abs

\def\S{{\mathcal{S}}}
\def\class{{\mathcal{C}}}
\def\J{{J}}
\def\G{{\mathcal{G}}}

\def\P{{\mathcal{P}}}
\def\pairings{{\mathcal{C}}}
\def\class{{\mathcal{S}}}
\def\nonnegints{{\mathbb{N}}}
\let\eps=\varepsilon

\def\({\bigl(}   \def\){\bigr)}
\def\abs#1{\mathopen|#1\mathclose|} \let\card=\abs

\def\kvec{{\boldsymbol{k}}}
\def\ff#1#2{[#1]_{#2}}
\def\ac#1{\card{\pairings_{#1}}}
\def\dfrac#1#2{\lower0.15ex\hbox{\large$\frac{#1}{#2}$}}

\def\kmax{k_{\mathrm{max}}}
\def\Gp{{G_p}}
\def\Gnaive{{G_\mathrm{naive}}}

\def\kbar{k}
\def\nfrac#1#2{{\textstyle\frac{#1}{#2}}}
\def\dfrac#1#2{\lower0.15ex\hbox{\large$\frac{#1}{#2}$}}
\let\originalleft\left
\let\originalright\right
\renewcommand{\left}{\mathopen{}\mathclose\bgroup\originalleft}
\renewcommand{\right}{\aftergroup\egroup\originalright}

\title{Asymptotic enumeration of sparse multigraphs\\
 with given degrees}

\author{
Catherine Greenhill
\thanks{Research supported by the Australian Research Council.} \\
\small School of Mathematics and Statistics\\[-0.8ex]
\small The University of New South Wales\\[-0.8ex]
\small Sydney NSW 2052, Australia\\
\small \tt csg@unsw.edu.au\\
\and
Brendan D. McKay
\footnotemark[\value{footnote}]\\
\small Research School of Computer Science\\[-0.8ex]
\small Australian National University\\[-0.8ex]
\small Canberra, ACT 0200, Australia\\
\small\tt bdm@cs.anu.edu.au
}

\begin{document}

\maketitle

\begin{abstract}
Let $\J$ and $\J^*$ be subsets of $\nonnegints$ such that
$0,1\in\J$ and $0\in\J^*$.
For infinitely many $n$, let $\kvec=(k_1,\ldots, k_n)$ be a vector of
nonnegative integers whose sum $M$ is even.
We find an asymptotic expression for the number of multigraphs on
the vertex set $\{ 1,\ldots, n\}$ with degree sequence given
by~$\kvec$, such that every loop has multiplicity in~$\J^*$
and every non-loop edge has multiplicity in~$\J$.
Equivalently, these are symmetric integer matrices with values
$J^*$ allowed on the diagonal and $J$ off the diagonal.
Our expression holds when the maximum degree $\kmax$
satisfies $\kmax = o(M^{1/3})$.  We prove this result using
the switching method, building on an asymptotic
enumeration of simple graphs
with given degrees (McKay and Wormald, 1991).
Our application of the switching method introduces a novel way 
of combining several different switching operations into a single
computation.
\end{abstract}

\section{Introduction}\label{s:introduction}

Multigraphs arise in many applications including modelling
transportation networks~\cite{BCFMR}, the structure of RNA~\cite{GFZSTLKS,KKRR}
and in nonparametric statistics~\cite{godehardt}.
We use terminology ``multigraphs'' inclusively, with both multiple edges and 
(possibly multiple) loops allowed. 
We seek an asymptotic enumeration formula for 
multigraphs with a given degree sequence,
satisfying certain conditions.

Let $k_{i,n}$ be a nonnegative integer for all pairs $(i,n)$ of
integers which satisfy $1\leq i\leq n$.  Then for each $n\geq 1$,
let $\kvec = \kvec(n) = (k_{1,n},\ldots, k_{n,n})$.
We usually write $k_i$ instead of $k_{i,n}$.
Define $M = \sum_{i=1}^n k_i$. 
We assume that $M$ is even for an infinite number of values of $n$,
and tacitly restrict ourselves to such $n$.

For subsets $J,J^*$ of $\nonnegints$, define $\G(\kvec,\J, \J^*)$
to be the set of all multigraphs on the vertex set $\{ 1,\ldots, n\}$
with degree sequence given by $\kvec$ such that the multiplicity of
every loop belongs to $\J^*$ and the multiplicity of every non-loop edge
belongs to~$\J$.
Loops contribute 2 to the degree of their vertex.
Also define $\G(\kvec)$ to be the set of multigraphs
on $\{ 1,\ldots, n\}$
with degree sequence $\kvec$ (and no restrictions on the
multiplicities).

Equivalently we may think of $\G(\kvec,\J,\J^*)$
as the set of all $n\times n$ symmetric matrices $A = (a_{ij})$
with all diagonal entries in
$\J^*$, all off-diagonal entries in $\J$ and with
\begin{equation}
 2 a_{ii} + \sum_{j\neq i} a_{ij} = k_i
\label{rowsum}
\end{equation}
for $i=1,2,\ldots, n$.  Note that diagonal entries are weighted by~2
in the row sum.

Let $\kmax = \max_{i=1}^n k_i$.  
We find an asymptotic expression for 
\[ G(\kvec,\J,\J^*) = 
      |\G(\kvec,\J,\J^*)|\]
that holds when $\kmax = o(M^{1/3})$.
For $r=1,2,\ldots$ let
\[ M_r = \sum_{i=1}^n \, [k_i]_r\]
where $[a]_r = a(a-1)(a-2)\cdots (a-r+1)$ denotes the falling factorial.
Then $M_1 = M$ and $M_r \leq \kmax\, M_{r-1}$ for all $r\geq 2$.

For $i\ge 0$, define 
\[ x_i= \begin{cases}  \,1, & \text{ if $i\in J$,}\\
                      \, 0, & \text{ otherwise;}
  \end{cases}
  \qquad
  y_i= \begin{cases}  \,1, & \text{ if $i\in J^*$,}\\
                       \,0, & \text{ otherwise.}
  \end{cases}
\] 
Our main result is the following.

\begin{theorem}\label{main}
Let $\J$ and $\J^*$ be subsets of $\nonnegints$ such
that $0,1\in\J$ and $0\in\J^*$.
Suppose that $n\to\infty$, $M\to\infty$ and $\kmax^3=o(M)$. Then
 \begin{align*}
    G&(\kvec,\J,\J^*) \\
     &= \frac{M!}{(M/2)!\,2^{M/2} k_1!\cdots k_n!} 
        \exp\biggl( \(y_1-\tfrac12\)\frac{M_2}{M}
          + \(x_2-\tfrac12\)\frac{M_2^2}{2M^2} 
          + \frac{M_2^4}{4M^5}   \\
         &{\hskip 14em} - \frac{M_2^2M_3}{2M^4}
          + \(x_3-x_2+\tfrac13) \frac{M_3^2}{2M^3}  
                    + O(\kmax^3/M)\biggr).
 \end{align*}
\end{theorem}

For convenience, we restate the formula in the regular case.

\begin{corollary}\label{regular}
Suppose that $\kvec=(k,\ldots,k)$  with $kn$ even,
such that $k=o(n^{1/2})$ as $n\to\infty$.  Then
  \[
     G(\kvec,\J,\J^*) = \frac{(kn)!}{(kn/2)!\, 2^{kn/2} (k!)^n} \,
           \exp\( - Q(k,n) + O(k^2/n)\),
  \]
  where   
\[   
  Q(k,n) = \nfrac{1}{4}\, (k-1)\((-1)^{x_2}(k-1) +2(-1)^{y_1}\) 
                 + \frac{k^3}{12n} (6x_2 - 6x_3 + 1).
\]
\end{corollary}

Setting $y_1=x_2=x_3=0$ we recapture the asymptotic expression for
sparse simple graphs with given degrees presented in~\cite[Theorem 5.2]{MW91}.
Similarly, setting $y_1=1$, $x_2=x_3=0$ we obtain the asymptotic enumeration
by degree sequence of sparse 
graphs with loops allowed but no multiple edges: this is the 
second expression in~\cite[Theorem 1.5]{Loopy}.

We remark that the conditions $0\in J^*$ and $0,1\in J$ in Theorem~\ref{main}
can be replaced by weaker conditions: namely, that $J^*$ is
non-empty and that $J$ has at least two elements with the
smallest two consecutive.  Let $s$ denote the least element of
$J^*$ and let $t, t+1$ be the smallest two elements of $J$.
This case is reduced to ours by
subtracting $s$ from each diagonal element, $t$ from each
off-diagonal element, and $2s+(n-1)t$ from each $k_i$,
provided the new degree sequence satisfies the
conditions of Theorem~\ref{main}.

Theorem~\ref{main} is proved using
the switching method, building on an asymptotic
enumeration of simple graphs with given degrees~\cite{MW91}.  
Our application of the switching method introduces a novel way 
of combining several different switching operations into a single
computation.

The remainder of the paper is structured as follows.  The history of this asymptotic enumeration
problem is briefly reviewed in Section~\ref{ss:history}, then the new 
switching theorem is
given in Section~\ref{s:switchinglemma}.  In Section~\ref{s:multigraphs} we describe 15 types of
switchings on multigraphs, which are used to show that certain
multiplicities of edges or loops are rare.  These switchings are analysed together
in Subsection~\ref{ss:combine} using the new switching theorem.   
In Section~\ref{s:pairings} we complete the enumeration
with the help of some calculations performed in~\cite{MW91}.

Finally, in Section~\ref{s:naive} we show that a na\"\i ve argument leads to a
formula for $G(\kvec,\J,\J^*)$ which differs asymptotically
by a constant factor
from the result of Theorem~\ref{main} in the regular case. 
The constant factor takes two different values, depending on whether $2\in J$.

\subsection{History}\label{ss:history}

The earliest work on this problem was that of Read~\cite[p.\,156]{Read},
who found exact and asymptotic formulae when
$k_1=\cdots=k_n=3$, for all four combinations of 
$J^*\in\{ \{0\},\,\nonnegints\}$ and $J\in\{\{0,1\},\, \nonnegints\}$.
The best result for $J^*=\{0\},\allowbreak
J=\{0,1\}$ in the sparse range is
that of McKay and Wormald~\cite{MW91}
who treated $\kmax=o(M^{1/3})$; see that paper for
a survey of the many earlier results on the \hbox{0-1} case.
In addition to $J^*=\{0\},J=\{0,1\}$, Bender and Canfield~\cite{BC78}
found the asymptotics for $J^*=\{0\},J=\nonnegints$ when $\kmax=O(1)$.
(Although that paper allows nonzero diagonal entries, they contribute
singly to the row sums, not doubly as we have~it.)

In~\cite{MW90}, McKay and Wormald considered $J^*=\{0\},J=\{0,1\}$
in the dense domain defined by $\min\{k,n-k-1\}\ge c n/\log n$
for $c>\nfrac23$ and $\abs{k_i-k}=O(n^{1/2+\eps})$ for
all~$i$, where $k$ is the average degree.
McKay~\cite{ranX} found a better error term under the same
conditions, while Barvinok and Hartigan allowed a wider range
of degrees~\cite{BH12}.
Greenhill and McKay~\cite{Loopy} added the option of
loops in both the sparse and dense ranges by considering
$J^*=\{0,1\}$,  $J=\{0,1\}$.
Finally, McLeod and McKay~\cite{MM11} analysed the case of
$J^*=\{0\}, J=\nonnegints$ when $k_1=\cdots=k_n>cn/\log n$ for
$c>\nfrac16$.

For sparse rectangular matrices which are not necessarily symmetric,
see Greenhill and McKay~\cite{GM08}.

\section{The switching theorem}\label{s:switchinglemma}

In order to bound the number of multigraphs having some unusual 
properties, we will apply the method of switchings.  It will be
necessary to apply several switching types, which we could
analyse one at a time as in~\cite{GM08}.  Instead
we now prove a generalized switching theorem that allows us
to analyse all of them at once.

In order to facilitate use of the method in future work,
we will present the theorem in greater generality than is
required in this paper.

\begin{theorem}\label{switchings}
Let $\varGamma=\varGamma(V,E)$ be a directed multigraph,
and let $\alpha:E\to \mathbb{R}_+$ be a positive weighting
of the edges of $\varGamma$.  Fix a non-empty finite set $C$, whose elements we will
call \textit{colours}, and let $c:E\to C$ be an edge colouring of $\varGamma$  (which need not be proper).
For all $v\in V$ and $c\in C$ denote by $\varGamma_c^-(v)$ the set of
edges of colour~$c$ entering~$v$, and by $\varGamma_c^+(v)$ the
set of edges of colour $c$ leaving~$v$.

We introduce the set of variables $\{N(v) : v\in V\}\cup\{s(e) : e\in E\}$
and consider the following system
of linear inequalities on these variables:
\begin{align}
   N(v) \ge 0,& \qquad (v\in V) \label{ineq1} \\
   s(e) \ge 0,& \qquad (e\in E) \label{ineq2}\displaybreak[0]  \\
   \sum_{e\in\varGamma_c^-(v)} s(e) \le N(v),&
     \qquad(v\in V,\,\, c\in C) \label{ineq3}\\
   \sum_{e\in\varGamma_c^+(v)} \alpha(e)s(e) \ge N(v),&
     \qquad(v\in V,\,\, c\in C,\,\, \varGamma_c^+(v)\ne\emptyset). \label{ineq4}
\end{align}
Let $C(v)=\{c\in C :\varGamma_c^-(v)\ne \emptyset \}$ be the
set of colours entering~$v$.
For each $v\in V$ which is not a source, 
let $\lambda_c(v), c\in C(v)$
be positive numbers such that
$\sum_{c\in C(v)}\lambda_c(v)\le 1$.
For each $vw\in E$, define $\hat\alpha(vw)=\alpha(vw)/\lambda_{c(vw)}(w)$.
Extend this function to 
any directed path $P$ by defining $\hat\alpha(P)=\prod_{e\in P}\hat\alpha(e)$.

Now suppose that $Y,Z\subseteq V$ satisfy the following conditions:
\begin{itemize}
\item[1.] $Z\ne \emptyset$ and $Y\cap Z=\emptyset$;
\item[2.] If $v\in V$ is a sink of $\varGamma$, or if $\hat\alpha(vw)\ge 1$
   for some $vw\in E$, then $v\in Z$.
\end{itemize}
For any $W,W'\subseteq V$, define $\P(W,W')$ to be the set of 
non-trivial directed paths in $\varGamma$ which start in $W$, 
end in $W'$, and have no internal vertices in $Y\cup Z$.
Then every solution of~\eqref{ineq1}--\eqref{ineq4} satisfies
\begin{equation}\label{pathineq}
  \sum_{v\in Y} N(v) \le
  \frac{\max_{P\in\P(Y,Z)}\hat\alpha(P)}{1-\max_{P\in\P(Y,Y)}\hat\alpha(P)}
  \sum_{v\in Z} N(v),
\end{equation}
where the maximum over an empty set is taken to be~0.
\end{theorem}

\begin{proof}
 The case of one colour appears in~\cite[Theorem 3]{HM}, apart from an
 inconsequential difference in the conditions on~$Z$.  We proceed
by reducing the general problem to one in which there is only
one colour, then modifying $\varGamma$
slightly so that~\cite[Theorem 3]{HM} applies.

Define $\varGamma^-(v)$ and
 $\varGamma^+(v)$ to be the set of all edges (regardless of colour)
 entering~$v$ or leaving~$v$, respectively.
 For each edge $vw\in E$, define $\hat s(vw)=s(vw)\lambda_{c(vw)}(w)$.
 If we weight inequality~\eqref{ineq3} by $\lambda_c(v)$ and sum
 over $c\in C(v)$, we obtain
 \begin{equation}
 \sum_{e\in\varGamma^-(v)} \hat s(e) \le N(v),
      \qquad(v\in V).
\label{1colour1}
 \end{equation}
 Similarly, if $\varGamma^+(v)\ne\emptyset$, then we can sum
 \eqref{ineq4} over those $c\in C$ with
 $\varGamma_c^+(v)\ne\emptyset$ to obtain
 \begin{equation}
   \sum_{e\in\varGamma^+(v)} \hat\alpha(e)\hat s(e) \ge N(v),
        \qquad(v\in V,\,\, \varGamma^+(v)\ne\emptyset).
 \label{1colour2}
\end{equation}
Together with the nonnegativity of $N(v)$ and $\hat s(e)$,
we have equations of the form of \eqref{ineq1}--\eqref{ineq4}
with only one colour.

Next remove from $\varGamma$ all edges $vw$ where $v\in Z$.
Doing so can only weaken the conditions, by decreasing the left
hand side of some inequalities in (\ref{1colour1}) 
or removing some inequalities in 
(\ref{1colour2}). 
Note that none of the quantities in~\eqref{pathineq} are changed
by removal of these edges.
After this change to $\varGamma$, $Z$ satisfies the requirements of~\cite[Theorem 3]{HM}, with the variable $X$ defined there
set equal to~$Z$.
Applying that theorem to the system defined by
(\ref{1colour1}), (\ref{1colour2}) and the nonnegativity of
$N(v)$ and $\hat s(e)$ gives~\eqref{pathineq}.~%
\end{proof}

We now describe how Theorem~\ref{switchings} can be used for
counting.

Suppose we have a finite set of ``objects'' partitioned into disjoint
classes $\class(v)$, where $v\in V$ for some index set $V$.
Define $N(v)=\card{\class(v)}$ for each $v\in V$. 
Also suppose that for each $c\in C$ we have a relation
$\varPsi_c$ between objects: to be precise, $\varPsi_c$ is
a multiset of ordered pairs $(Q,R)$ of objects.
(We call $\varPsi_c$ a \textit{switching} and usually define it by
some operation that modifies $Q$ to make~$R$.)

Now define an edge-coloured directed multigraph $\varGamma=(V,E)$
with vertex set~$V$, where
$\varGamma$ has a directed edge $vw$ of colour $c$ if and only if
$(Q,R)\in \varPsi_c$ for some $Q\in\class(v)$ and $R\in\class(w)$.
(There is at most one edge of each colour between any pair
of distinct vertices of $\varGamma$.) 
For each $vw\in E$ let
$s'(vw) = \card{\{(Q,R)\in \varPsi_c : Q\in\class(v),R\in\class(w)\}}$,
counting multiplicities, where $c$ is
the colour of~$vw$.

Fix $v\in V$ and $c\in C$ such that $\varGamma_c^+(v)\neq \emptyset$.
Suppose that for any $Q\in\class(v)$ there are at least
$a_c(v)>0$ objects $R$ with $(Q,R)\in \varPsi_c$, counting
multiplicities.  Then
\[
    \sum_{e\in\varGamma_c^+(v)} s'(e)\ge a_c(v)N(v).
\]
Similarly,  for fixed $w\in V$ and $c\in C$, suppose that for every $R\in\class(w)$
there are at most $b_c(w)>0$ objects $Q$ with $(Q,R)\in \varPsi_c$,
counting multiplicities. Then
\[
    \sum_{e\in\varGamma_c^-(w)} s'(e)\le  b_c(w)N(w).
\]
Defining $s(vw)=s'(vw)/b_{c(vw)}(w)$
and $\alpha(vw)=b_{c(vw)}(w)/a_{c(vw)}(v)$ we obtain
equations~\eqref{ineq1}--\eqref{ineq4}.
Theorem~\ref{switchings} can thus be used to bound the relative values
of $\sum_{v\in Z} \card{\class(v)}$ and $\sum_{v\in Y} \card{\class(v)}$ if
$Y,Z$ satisfy the requirements of the lemma.
Since $\sum_{v\in Z} \card{\class(v)}\le\sum_{v\in V} \card{\class(v)}$,
this also bounds $\sum_{v\in Y} \card{\class(v)}$ relative to
$\sum_{v\in V} \card{\class(v)}$; i.e., it bounds the fraction of
all objects that lie in $\bigcup_{v\in Y} \S(v)$.

\section{Switchings on multigraphs}\label{s:multigraphs}

Define
\begin{align}
\begin{split}
N_1 &= \max\bigl\{ \lceil\log M\rceil, \lceil 480 M_2/M\rceil\bigr\},\\
N_2 &= \max\bigl\{ \lceil\log M\rceil, \lceil 240 M_2^2/M^2\rceil\bigr\},\\
N_3 &= \max\bigl\{ \lceil\log M\rceil, \lceil 240 M_3^2/M^3\rceil\bigr\}.
\end{split}
\label{ndefs}
\end{align}
(We have not attempted to optimise constants.)
Given a multigraph $Q$, let $\ell_D(Q)$ denote the number of loops
with multiplicity $D$ and let $e_D(Q)$ denote the number of non-loop edges
with multiplicity $D$, for $D\geq 1$.

Let
\[ 
\G_0 = \G(\kvec,\J\cup \{4,5,6,\ldots\},
                            \J^*\cup\{2,3,4,\ldots\})\\
\]
be the set of all multigraphs with degree sequence $\kvec$, and allowing all 
multiplicities except for those in $\{ 1\} - \J^*$ on loops and those
in $\{2,3\}-\J$ on non-loops.  Note that
$\G(\kvec,\J,\J^*)\subseteq \G_0$.

We also define the subsets $Y$, $Z$ of $\G_0$ by
\begin{align}
 Y &= \G_0 - \{ Q\in \G_0 \mid \ell_D(Q) =0 \text{ for $D\geq 2$, }\,
            e_D(Q) = 0 \text{ for $D\geq 4$, }\, \notag\\
   & \hspace*{4cm} 
                   \ell_1(Q)\leq N_1,\,\,
            e_2(Q) \leq N_2,\,\, e_3(Q)\leq N_3\}, \label{Ydef} \\
Z &= \{ Q\in \G_0 \mid \ell_D(Q) =0 \text{ for $D\geq 2$, }\,
            e_D(Q) = 0 \text{ for $D\geq 4$, }\notag\\
   & \hspace*{3cm}  
                   \ell_1(Q)\leq \lceil N_1/2\rceil,\,\,
         e_2(Q) \leq \lceil N_2/2\rceil,\,\, e_3(Q)\leq \lceil N_3/2\rceil\}.
\label{Zdef}
\end{align}

We define 15 coloured switchings which act to reduce the
number of loops or edges with high multiplicities, moving
in steps from $Y$ towards $Z$.  These switchings are defined
below, together with a description of when each should be used.
Indeed, any given switching  will only be used on multigraphs in
$\G_0$ for which
none of the  switchings with a lower-labelled colour 
are applicable.
An important property of all the switchings is that they do not
create simple loops (with multiplicity 1) and they do not create
non-loop edges of multiplicity 2~or~3, as these may not be allowed
for multigraphs in~$\G_0$.

For each switching colour $c$ and multigraphs $Q$, $R$ 
we will define the following parameters:  

\vspace{-\topsep}
\begin{itemize} \itemsep=0pt
\item $a_c(Q)$ is a lower bound on the number of ways in which a
switching of colour $c$ can be applied to $Q$.
\item $b_c(R)$ is an upper bound on the number of ways in which a
switching of colour $c$ can produce $R$.
\item If $R$ can be obtained from $Q$ by performing
a switching of colour $c$, then we let $\alpha(Q,R)$ be an upper bound on
$b_c(R)/a_c(Q)$.  (Note that the colour $c$ is determined by the
edge~$QR$.)
\end{itemize}
Finally in Section~\ref{ss:combine} we will analyse all the switchings
at once by applying Theorem~\ref{switchings}.

\subsection{Switchings of colour 1}\label{ss:switching1}

A switching of colour 1 is used to reduce the number of loops of 
multiplicity equal to 2 or multiplicity at least 4. It is applied
to multigraphs $Q\in \G_0$ with $L(Q) > \lceil 3M^{1/2}\rceil$, where 
\begin{equation}
\label{ldef}
 L(Q) = \ell_2(Q) + \sum_{D\geq 4}\ell_D(Q)
\end{equation}
is the number of loops in $Q$ with multiplicity 2 or
multiplicity at least 4.
The switching is described by the sequence $(v_1,v_2)$
of distinct vertices such that 
there is a loop of multiplicity $D_1$ at $v_1$ and a 
loop of multiplicity $D_2$ at $v_2$, with 
$D_1, D_2\in \{2\}\cup\{4,5,\ldots\}$.  

Let $m$ be the multiplicity of the edge
from $v_1$ to $v_2$ in $Q$ (which may equal zero).  The switching reduces
the multiplicity of both loops by 2, and increases the multiplicity of 
the edge $\{ v_1,v_2\}$ by 4.  This operation is depicted in
Figure~\ref{f:switching1}.

\begin{figure}[ht]
\begin{center}
\unitlength=0.9cm
\begin{picture}(9,1.5)(0,0)
\put(0,0.0){
    \includegraphics[scale=0.72]{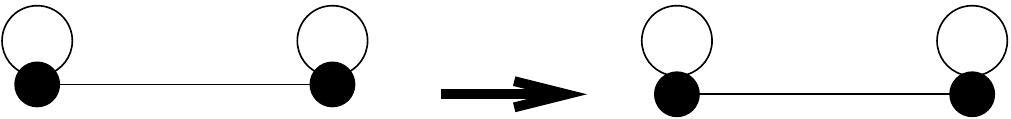}
        }
\put(0.2,1.15){$D_1$}
\put(2.6,1.15){$D_2$}
\put(4.9,1.15){$D_1-2$}
\put(7.3,1.15){$D_2-2$}
\put(1.5,0.5){$m$}
\put(6.3,0.4){$m+4$}
\end{picture}
\caption{A switching of colour 1}
\label{f:switching1}
\end{center}
\end{figure}
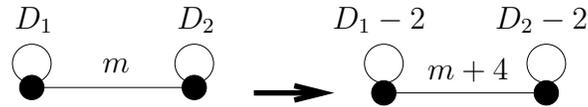

Suppose that $R$ can be produced from $Q$ using a switching of colour 1.
The number of ways to perform a switching of colour 1 in $Q$ is exactly
\[ [ L(Q) ]_{2} \geq 9M \]
as there is no restriction on the value of $m$.  The number of
ways that $R$ can be produced using a switching of colour 1
is bounded above by $M/4$, which is a bound on the number of ways to choose
an oriented edge with multiplicity at least 4.  
Hence we can set $a_1(Q) = 9M$ and $b_1(R) = M/4$, giving
\[ \alpha(Q,R) = \dfrac{1}{36}. \]

\subsection{Switchings of colour 2}\label{ss:switching2}

A switching of colour 2 is used to reduce the number of loops of 
multiplicity three to at most $\lceil 3M^{1/2}\rceil$. 
It is performed for multigraphs $Q$ for which switchings of colour 1
do not apply and such that $\ell_3(Q) > \lceil 3M^{1/2}\rceil$.
The switching is described by the sequence $(v_1,v_2)$ of
distinct vertices such that 
there is a loop of multiplicity 3 at $v_1$ and at $v_2$.

Let $m$ be the multiplicity of the edge from $v_1$ to $v_2$ in $Q$
(which may equal zero).  The switching removes the loop at $v_1$
and at $v_2$, and increases the multiplicity of the edge $\{ v_1,v_2\}$
by 6, as illustrated in Figure~\ref{f:switching2}. 

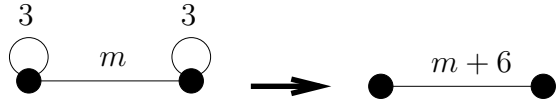
\begin{figure}[ht]
\begin{center}
\unitlength=0.9cm
\begin{picture}(9,1.5)(0,0)
\put(0,0.0){
    \includegraphics[scale=0.72]{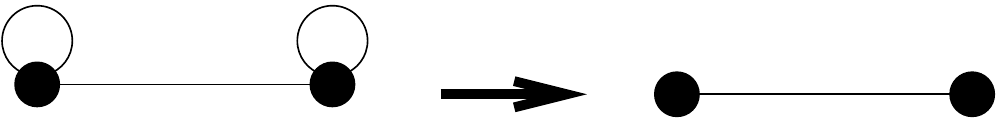}
        }
\put(0.3,1.1){$3$}
\put(2.7,1.1){$3$}
\put(1.5,0.5){$m$}
\put(6.4,0.4){$m+6$}
\end{picture}
\caption{A switching of colour 2}
\label{f:switching2}
\end{center}
\end{figure}

Suppose that the multigraph $R$ can be produced from $Q$ using a
switching of colour 2.  
The number of ways to perform a switching of colour 2 in $Q$ is 
exactly $[\ell_3(Q)]_2 \geq 9M$, and the number of ways that $R$ can
be produced using a switching of colour 2 is at most $M/6$
(since an oriented edge of multiplicity at least 6 determines the
reverse operation, from $R$ to $Q$).
Hence we can take
$a_2(Q)= 9M$ and $b_2(R) = M/6$, leading to
\[ \alpha(Q,R) = \dfrac{1}{54}. \]

\subsection{Switchings of colour 3}\label{ss:switching3}

A switching of colour 3 is used to reduce the number of loops of
multiplicity 1 to at most $\lceil M^{1/2}\rceil$.
It is applied to multigraphs $Q$ for which switchings of colour 1 and 2
do not apply
and such that $\ell_1(Q) > \lceil M^{1/2}\rceil$.
The switching is defined by the sequence
$(v_1,v_2,v_3)$ of distinct vertices such that each of $v_1$, $v_2$,
$v_3$ has a simple loop in $Q$ (of multiplicity 1), and
none of the edges $\{ v_1,v_2\}$, $\{ v_1,v_3\}$,
$\{ v_2,v_3\}$ are present in $Q$.
The switching removes these three loops and joins the three vertices
pairwise by simple edges (of multiplicity 1), as illustrated in
Figure~\ref{f:switching3}.

\begin{figure}[ht]
\begin{center}
\unitlength=0.9cm
\begin{picture}(9,3.8)(0,0)
\put(0,-0.3){
    \includegraphics[scale=0.72]{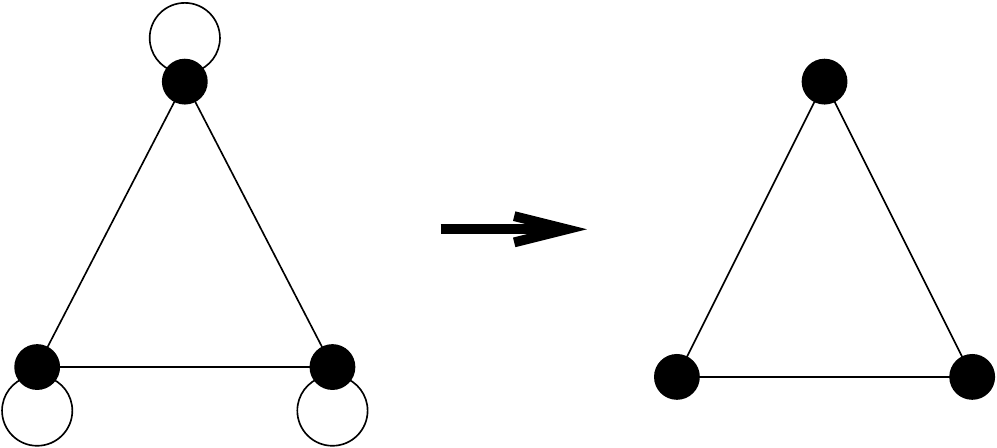}
        }
\put(-0.2,0.0){$1$}
\put(3.3,0.0){$1$}
\put(1.5,3.5){$1$}
\put(0.6,1.5){$0$}
\put(2.5,1.5){$0$}
\put(1.5,-0.2){$0$}
\put(5.8,1.5){$1$}
\put(7.7,1.5){$1$}
\put(6.7,-0.2){$1$}
\end{picture}
\caption{A switching of colour 3}
\label{f:switching3}
\end{center}
\end{figure}
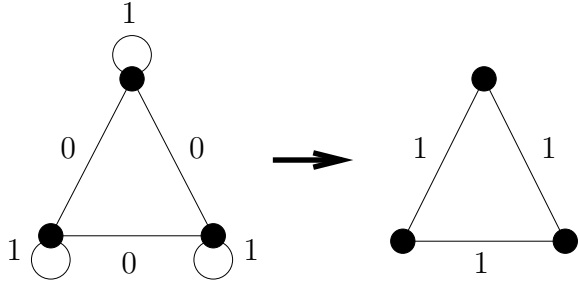

Suppose that the multigraph $R$ can be produced from $Q$ using a switching
of colour 3.
The number of ways to perform a switching of colour 3 in $Q$ is at least
\[ [\ell_1(Q)]_3 - O\(\kmax  \ell_1(Q)^2\) = \ell_1(Q)^3(1-o(1)) \geq
  \nfrac{1}{2}\, M^{3/2},\]
while the number of ways that $R$ can be produced using a switching
of colour 3 is at most $\kmax M$.
Therefore we can take $a_3(Q) = \nfrac{1}{2} M^{3/2}$ and $b_3(R) = \kmax M$,
leading to
\[ \alpha(Q,R) = \frac{2\kmax}{M^{1/2}} = o(1).\]

\subsection{Switchings of colour 4}\label{ss:switching4}

A switching of colour 4 is used to reduce the number of non-loop edges
of multiplicity greater than $\max\{4,\lceil \kmax^{1/2}\rceil\}$.  
It is applied to multigraphs $Q$ for which switchings of colours
$1,2,3$ do not apply and such that
$E^+(Q) > \lceil 4\kmax^{1/2}\, M^{1/2}\rceil$, where
\begin{equation} 
E^+(Q) = \sum_{D= \max\{ 4,\,\lceil \kmax^{1/2}\rceil\}+1}^{\kmax}  e_D(Q).
\label{EQ-definition}
\end{equation}
A switching of colour 4 in $Q$ is described by a sequence
$(v_1,w_1,v_2,w_2)$ of distinct vertices such that

\vspace{-\topsep}
\begin{itemize} \itemsep=0pt
\item
 the multiplicity of edge $(v_1,w_1)$ in $Q$ is
$D_1$ and the multiplicity of $(v_2,w_2)$ in $Q$ is $D_2$,
where $D_1,D_2> \max\{4,\lceil \kmax^{1/2}\rceil\}$,
\item the edges $\{ v_1,v_2\}$ and $\{ w_1,w_2\}$ have multiplicity
zero in $Q$ (they are non-edges).
\end{itemize}
The switching reduces the multiplicity of edges $(v_1,w_1)$ and
$(v_2,w_2)$ by one and gives multiplicity 1 to edges
$\{ v_1,v_2\}$ and $\{w_1,w_2\}$, as shown in Figure~\ref{f:switching45}.

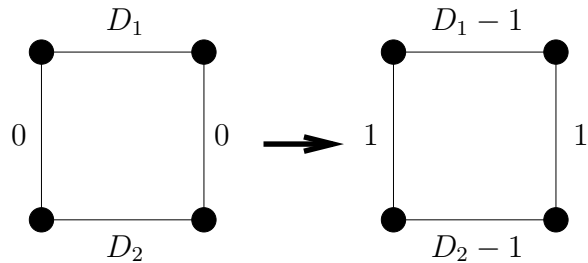
\begin{figure}[ht]
\begin{center}
\unitlength=0.9cm
\begin{picture}(8.5,4.0)(0,-0.4)
\put(0.0,0.0){
    \includegraphics[scale=0.72]{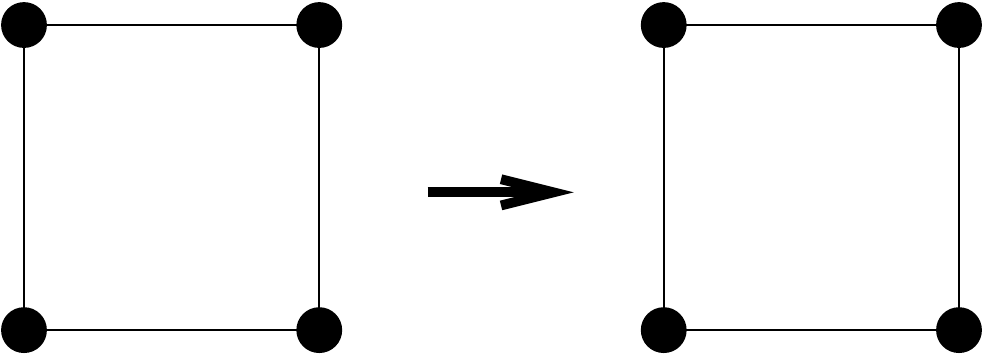}
        }
\put(1.3,3.0){$D_1$}
\put(-0.1,1.3){$0$}
\put(1.3,-0.35){$D_2$}
\put(2.9,1.3){$0$}
\put(6.1,3.0){$D_1-1$}
\put(5.1,1.3){$1$}
\put(6.1,-0.35){$D_2-1$}
\put(8.2,1.3){$1$}
\end{picture}
\caption{A switching of colour 4 or 5}
\label{f:switching45}
\end{center}
\end{figure}

Now suppose that the multigraph $R$ can be produced from $Q$ using a
switching of colour~4.  
The number of switchings of colour 4 that can be performed in $Q$
is  bounded below by
\[ 4 [E^+(Q)]_2 - O(\kmax^2\, E^+(Q) ) = 4E^+(Q)^2(1-o(1)) \geq 
            64\, \kmax\, M (1-o(1)) > 60\kmax\, M,\]
and the number of switchings of colour 4 that can produce $R$ is at most
\[ \frac{2 \kmax^2 M}{\max\{ 4,\, \lceil \kmax^{1/2}\rceil\}^2} 
          \leq 2 \kmax\, M.\]
Hence we can take $a_4(Q) = 60\,\kmax\, M$ and $b_4(R) = 2\kmax M$,
leading to
\[ \alpha(Q,R) = \dfrac{1}{30}.\]

\subsection{Switchings of colour 5}\label{ss:switching5}

A switching of colour 5 is used to reduce the number of non-loop
edges of multiplicity at least 5 and at most $\lceil \kmax^{1/2}\rceil$.
For a multigraph $Q\in \G_0$, let
\[ E^-(Q) = \sum_{D=5}^{\lceil\kmax^{1/2}\rceil} e_D(Q).\]
Switchings of colour 5 are applied to multigraphs $Q$ 
for which switchings of colours $1,\ldots, 4$  do not apply and such that
$E^-(Q) > \lceil 3 \kmax M^{1/2}\rceil$.

A switching of colour 5 is described by a sequence $(v_1,w_1,v_2,w_2)$
of distinct vertices such that 

\vspace{-\topsep}
\begin{itemize} \itemsep=0pt
\item the multiplicity of $\{ v_1,w_1\}$
in $Q$ is $D_1$ and the multiplicity of $\{ v_2,w_2\}$ in $Q$
is $D_2$, where
$D_1,D_2\in \{ 5,6,\ldots, \lceil \kmax^{1/2}\rceil\}$, and
\item the multiplicity of $\{ v_1,v_2\}$ and $\{ w_1,w_2\}$ in
$Q$ is zero (these are non-edges).
\end{itemize}
This switching is also illustrated by Figure~\ref{f:switching45},
but with different conditions on $D_1, D_2$ as above.

Suppose that $R$ can be produced from $Q$ using a switching of
colour 5.  The number of ways that a switching of colour 5 can be
performed in $Q$ is bounded below by
\[ 4[E^-(Q)]_2 - O(\kmax^2 E^-(Q)) = 4 E^-(Q)^2(1-o(1))
            \geq 30 \kmax^2 M,\]
and the number of ways to produce $R$ using a switching
of colour 5 is at most
\[ \kmax^2 M.\]
Hence we can let $a_5(Q) = 30\kmax^2 M$ and $b_5(R)=\kmax^2 M$,
leading to
\[ \alpha(Q,R) =  \dfrac{1}{30}.\]

\subsection{Switchings of colour 6, 7, 8}\label{ss:switching6}

A switching of colour $4+j$ is used to reduce the number of
edges of multiplicity $j$ to
at most $\lceil M^{5/6}\rceil$, for $j=2,3,4$.
Switchings of colour $4+j$ are applied to multigraphs $Q$
for which switchings of colours $1,\ldots, 3+j$  do not apply and
such that
\[ \quad e_j(Q) > \lceil M^{5/6}\rceil. \]
Given such a multigraph $Q$, a switching of colour $4+j$ is defined by
a sequence 
\[ (v_1,w_1,v_2,w_2,\ldots, v_j,w_j)\]
of distinct vertices such that

\vspace{-\topsep}
\begin{itemize} \itemsep=0pt
\item $(v_r, w_r)$ is an edge of multiplicity $j$
in $Q$ for $r=1,\ldots, j$,
\item every edge $\{ v_r, w_s\}$ with $1\leq r\neq s\leq j$ has
multiplicity 0 in $Q$ (it is a non-edge).
\end{itemize}
The switching deletes these $j$ edges of multiplicity $j$, 
and inserts a 
complete bipartite graph $K_{j,j}$ on 
$\{ v_1,\ldots, v_j\}\cup\{w_1,\ldots, w_j\}$, 
with all $j^2$ new edges simple (that is, multiplicity 1).  
This operation is illustrated in the Figure~\ref{f:switching8} for the case $j=4$.

\begin{figure}[ht]
\begin{center}
\unitlength=0.9cm
\begin{picture}(9,4.4)(0,0)
\put(0,0.0){
    \includegraphics[scale=0.72]{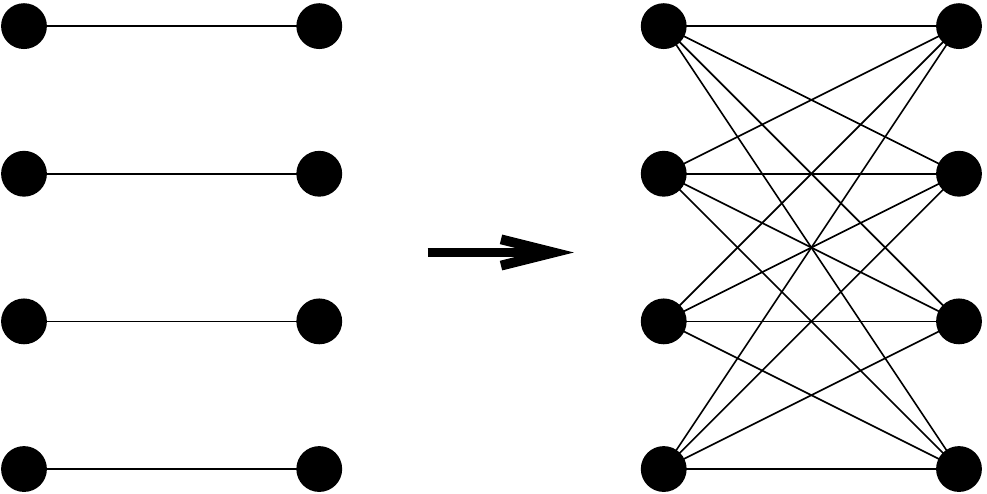}
        }
\put(1.4,0.4){$j$}
\put(1.4,2.8){$j$}
\put(1.4,4.0){$j$}
\put(1.4,1.6){$j$}
\end{picture}
\caption{A switching of colour 8}
\label{f:switching8}
\end{center}
\end{figure}
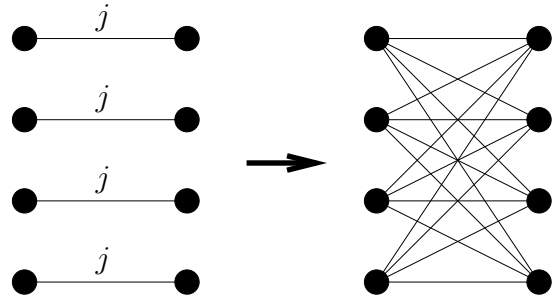

Suppose that the multigraph $R$ can be obtained from $Q$ using
a switching of colour $4+j$, for $j\in \{2,3,4\}$.
The number of ways in which a switching of colour $4+j$
can be performed in $Q$ is bounded below by
\[ 2^j\, [e_j(Q)]_j - O(\kmax^2\, e_j(Q)^{j-1}) \geq \nfrac{1}{2}  M^{5j/6},\]
while the number of switchings of colour $4+j$ which produce $R$
 is at most $\kmax^{2j-2} M$.
Hence we can define $a_{4+j}(Q) = \nfrac{1}{2} M^{5j/6}$ and
$b_{4+j}(R) = \kmax^{2j-2}\, M$ for $j\in\{2,3,4\}$.
Since $j\geq 2$, this leads to
\[
\alpha(Q,R) = \frac{2\kmax^{2j-2} M}{M^{5j/6}}
                  = O\left(\frac{\kmax^2}{M^{2/3}}\right) = o(1).\]

Before describing more coloured switchings, we prove a useful fact.

\begin{lemma}
\label{mostlyones}
Suppose that $Q\in\G_0$ is such that no switching of colour 1 to 8
applies to $Q$.  Then $e_1(Q)=(\tfrac12-o(1))M$.
\end{lemma}

\begin{proof}
Since no switching of colour 1, 2 or 3 applies we have
\[ \sum_{D\geq 1} D\,\ell_D(Q) = O(\kmax\, M^{1/2}) = O(M^{5/6}),\]
and since no switching of colour 4--8 applies we know that
\begin{align*}
 \sum_{D\geq 2} D\, e_D(Q) &\leq 9\lceil M^{5/6}\rceil +
   \kmax\, E^+(Q) + \lceil \kmax^{1/2}\rceil \, E^-(Q)\\
  &= O(M^{5/6} + \kmax^{3/2}\,M^{1/2})\\
  &= o(M).
\end{align*}
The result follows.  
\end{proof}

\subsection{Switchings of colour 9}\label{ss:switching9}

Switchings of colour 9 reduce the number of loops of
multiplicity 2 or multiplicity at least~4, until this number is zero.
They are 
applied to multigraphs $Q$ for which the switchings of
colours $1,\ldots, 8$ do not apply and such that $L(Q) \geq 1$.  

Let $Q$ be such a multigraph.  
A switching of colour 9 in $Q$ is defined by a sequence 
$(v_0,v_1,w_1,v_2,w_2)$ of distinct vertices such that

\vspace{-\topsep}
\begin{itemize} \itemsep=0pt
\item there is a loop at vertex $v_0$ in $Q$ with multiplicity $D$,
where $D=2$ or $D\geq 4$, 
\item $\{ v_1,w_1\}$ and $\{ v_2,w_2\}$ are simple edges (of multiplicity 1),
\item $\{ v_0,v_1\}$, $\{ v_0,v_2\}$, $\{ v_0,w_1\}$, $\{v_0,w_2\}$ are
all non-edges in $Q$ (with multiplicity zero).
\end{itemize}
The switching reduces the multiplicity of the loop at $v_0$
to $D-2$,
removes the two simple edges $\{v_1,w_1\}$
and $\{v_2,w_2\}$ and inserts the four simple edges
$\{ v_0,v_1\}$, $\{v_0,v_2\}$, $\{v_0,w_1\}$, $\{v_0,w_2\}$,
as shown in Figure~\ref{f:switching9}.

\begin{figure}[ht]
\begin{center}
\unitlength=0.9cm
\begin{picture}(8,3.4)(0,0)
\put(0,0.0){
    \includegraphics[scale=0.72]{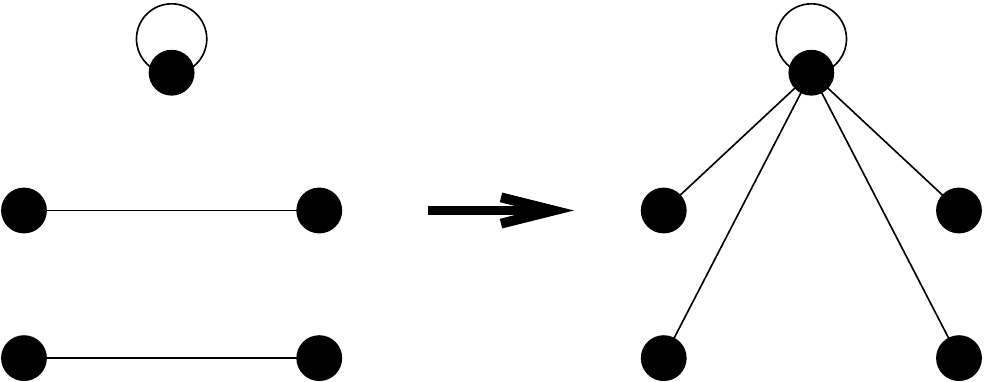}
        }
\put(1.4,3.3){$D$}
\put(6.2,3.3){$D-2$}
\end{picture}
\caption{A switching of colour 9}
\label{f:switching9}
\end{center}
\end{figure}

Suppose that the multigraph $R$ can be produced from $Q$ by a switching
of colour~9.  
The number of ways to perform a switching of colour 9 in $Q$ is at least
\[   L(Q)\left([2e_1(Q)]_2 - O(\kmax^2e_1(Q))\right) \geq \nfrac{1}{2}\, L(Q)M^2,\]
using Lemma~\ref{mostlyones},
while the number of ways to produce $R$ using a switching of
colour 9 is at most $M_4$.  (We ignore the presence of the loop at
$v_0$ in this upper bound, since no such loop exists when $D=2$.)

Hence we can let $a_9(Q) =\nfrac{1}{2} L(Q)M^2$ and 
$b_9(R) = M_4$, leading to 
\[ \alpha(Q,R) = \frac{2M_4}{L(Q)\, M^2} 
    = O\left(\frac{\kmax^3}{M} \right)\]
since $L(Q)\geq 1$.

\subsection{Switchings of colour 10}\label{ss:switching10}

Switchings of colour 10 reduce the number of loops of
multiplicity 3, until this number is zero.
They are applied to multigraphs $Q$ such that 
switchings of colours $1,\ldots, 9$ do not apply and such that 
$\ell_3(Q) \geq 1$.

Let $Q$ be such a multigraph.  A switching of colour 10 
is defined by a sequence of
distinct vertices $(v_0,v_1,w_1,v_2,w_2,v_3,w_3)$ such that

\vspace{-\topsep}
\begin{itemize} \itemsep=0pt
\item there is a loop of multiplicity 3 at $v_0$ in $Q$,
\item there is a simple edge $\{ v_j, w_j\}$ in $Q$, for $j=1,2,3$,
\item the edges $\{ v_0,v_j\}$ and $\{v_0,w_j\}$ all have
multiplicity 0 in $Q$, for $j=1,2,3$.
\end{itemize}
The switching removes the loop of multiplicity 3 and the simple
edges $\{ v_j,w_j\}$ for $j=1,2,3$, and adds the six simple
edges $\{ \{ v_0,v_j\},\, \{v_0,w_j\} \mid  j=1,2,3\}$,
as shown in Figure~\ref{f:switching10}.

\begin{figure}[ht]
\begin{center}
\unitlength=0.9cm
\begin{picture}(8,4.7)(0,0)
\put(0,0.0){
    \includegraphics[scale=0.72]{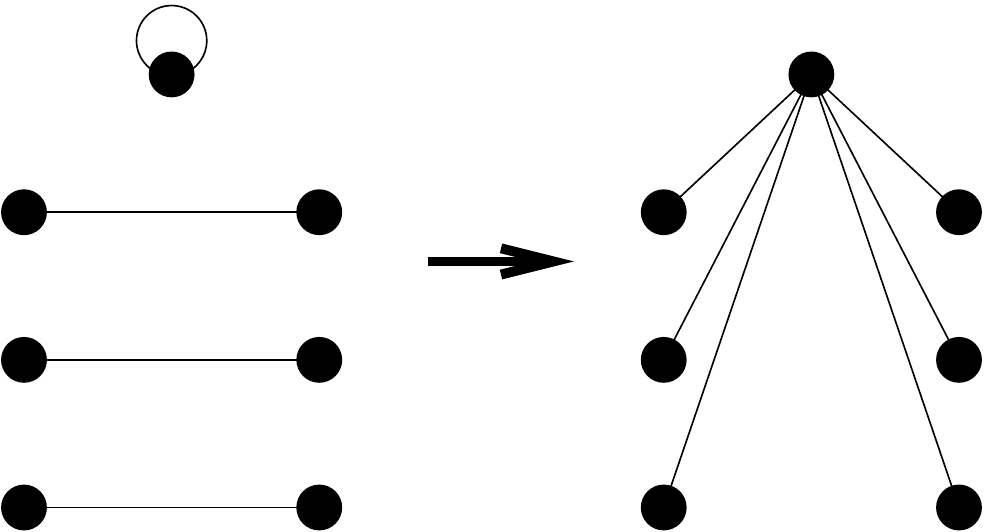}
        }
\put(1.4,4.5){$3$}
\end{picture}
\caption{A switching of colour 10}
\label{f:switching10}
\end{center}
\end{figure}

Now let $R$ be a multigraph which can be formed from $Q$ by a switching
of colour 10.  The number of ways that a switching of colour 10 can
be performed in $Q$ is at least
\[ \ell_3(Q)\, \left( [2e_1(Q)]_3 - O(\kmax^2 e_1(Q)^2)\right) \geq \nfrac{1}{2}\,  \ell_3(Q) M^3,\]
using Lemma~\ref{mostlyones},
and the number of ways that $R$ can be produced using a switching of
colour 10 is bounded above by $M_6$.  Hence we can take
$a_{10}(Q) = \nfrac{1}{2}\, \ell_3(Q)\, M^3$ and $b_{10}(R) = M_6$.
This leads to
\[ \alpha(Q,R) = \frac{2\, M_6}{\ell_3(Q)\, M^3}
              = O\left(\frac{\kmax^5}{M^2}\right) = o\left(\frac{\kmax^3}{M}
  \right).
\]

We do not tackle single loops yet.  Instead, the next two
switchings reduce non-loop edges of high multiplicity down to zero.

\subsection{Switchings of colour 11}\label{ss:switching11}

Switchings of colour 11 reduce the number of non-loop edges
with multiplicity 4 or multiplicity at least 7, until this number is
zero.  Switchings of colour 11 are applied to multigraphs $Q$ for which
switchings of colours $1,\ldots, 10$ do not apply and
$E(Q) \geq 1$, where
\[ E(Q) = e_4(Q) + \sum_{D\geq 7} e_D(Q).\]
Let $Q$ be such a multigraph.  A switching of colour 11 in $Q$ is
defined by a sequence 
\[ (v_0,w_0,v_1,w_1,v_2,w_2,v_3,w_3)\]
of distinct vertices such that

\vspace{-\topsep}
\begin{itemize} \itemsep=0pt
\item the edge $\{ v_0,w_0\}$ has multiplicity $D$ in $Q$,
where $D=4$ or $D\geq 7$,
\item edges $\{ v_j, w_j\}$ are simple edges in $Q$
(with multiplicity 1) for $j=1,2,3$,
\item the edges $\{ v_0,v_j\}$ and $\{w_0,w_j\}$ are non-edges
in $Q$
for $j=1,2,3$ (with multiplicity zero).
\end{itemize}
The switching reduces the multiplicity of $\{ v_0,w_0\}$
to $D-3$, removes the edges $\{ v_j,w_j\}$ for $j=1,2,3$
and adds the simple edges $\{ v_0,v_j\}$ and $\{ w_0,w_j\}$
for $j=1,2,3$.  This operation is illustrated in Figure~\ref{f:switching11}.

\begin{figure}[ht]
\begin{center}
\unitlength=0.9cm
\begin{picture}(9,4.7)(0,0)
\put(0,0.0){
    \includegraphics[scale=0.72]{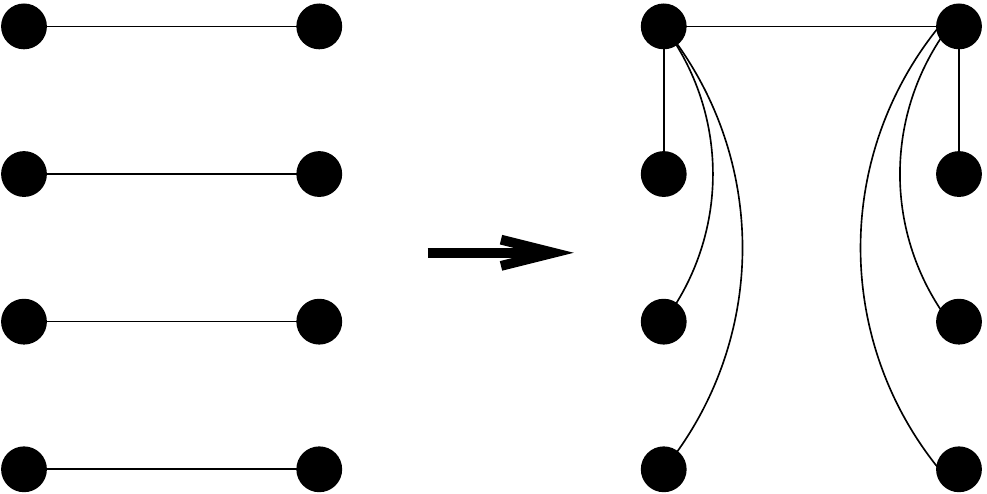}
        }
\put(1.3,4.1){$D$}
\put(6.2,4.1){$D-3$}
\end{picture}
\caption{A switching of colour 11}
\label{f:switching11}
\end{center}
\end{figure}

Suppose that the multigraph $R$ can be obtained from $Q$ by
performing a switching of colour 11.  The number of ways
to perform a switching of colour 11 in $Q$ is at
least
\[ 2 E(Q)\,\left( [2e_1(Q)]_3 - O(\kmax^2 e_1(Q)^2)\right) \geq  E(Q)\, M^3,\]
using Lemma~\ref{mostlyones}, and the number of ways to produce $R$ using a switching
of colour 11 is at most $\kmax^3 M_4$.  Hence we
can let $a_{11}(Q) =  E(Q)\, M^3$ and
$b_{11}(R) = \kmax^3 M_4$. This leads to
\[ \alpha(Q,R) = \frac{\kmax^3 M_4}{E(Q) M^3}
   = O\left(\frac{\kmax^6}{M^{2}}\right)
   =  o\left(\frac{\kmax^3}{M}\right).
\]

\subsection{Switchings of colour 12}\label{ss:switching12}

Switchings of colour 12 reduce the number of non-loop edges
with multiplicity 5 or 6, until this number is zero.
They are applied to multigraphs $Q$ such that switchings of
colours $1,\ldots, 11$ do not apply and $e_5(Q)+e_6(Q) \geq 1$.

Let $Q$ be such a multigraph.  Then a switching of colour 12
is defined by a sequence of distinct vertices
$(v_0,w_0,v_1,w_1,v_2,w_2,v_3,w_3,v_4,w_4,v_5,w_5)$
such that

\vspace{-\topsep}
\begin{itemize} \itemsep=0pt
\item the edge $\{ v_0,w_0\}$ has multiplicity $D$ in $Q$,
where $D\in \{ 5,6\}$,
\item each edge $\{ v_j,w_j\}$ is a simple edge in $Q$,
with multiplicity 1, for $j=1,2,3,4,5$,
\item the edges $\{ v_0,v_j\}$ and $\{ w_0,w_j\}$ all have
multiplicity 0 in $Q$ for $j=1,2,3,4,5$ (that is, they are
non-edges).
\end{itemize}
The switching reduces the multiplicity of the edge $\{ v_0,w_0\}$
to $D-5$, removes the edges $\{ v_j,w_j\}$, $j=1,2,3,4,5$,
and inserts the simple edges $\{ v_0,v_j\}$, $\{ v_0,w_j\}$
for $j=1,2,3,4,5$, as shown in Figure~\ref{f:switching12}.

\begin{figure}[ht]
\begin{center}
\unitlength=0.9cm
\begin{picture}(9,7.1)(0,0)
\put(0,0.0){
    \includegraphics[scale=0.72]{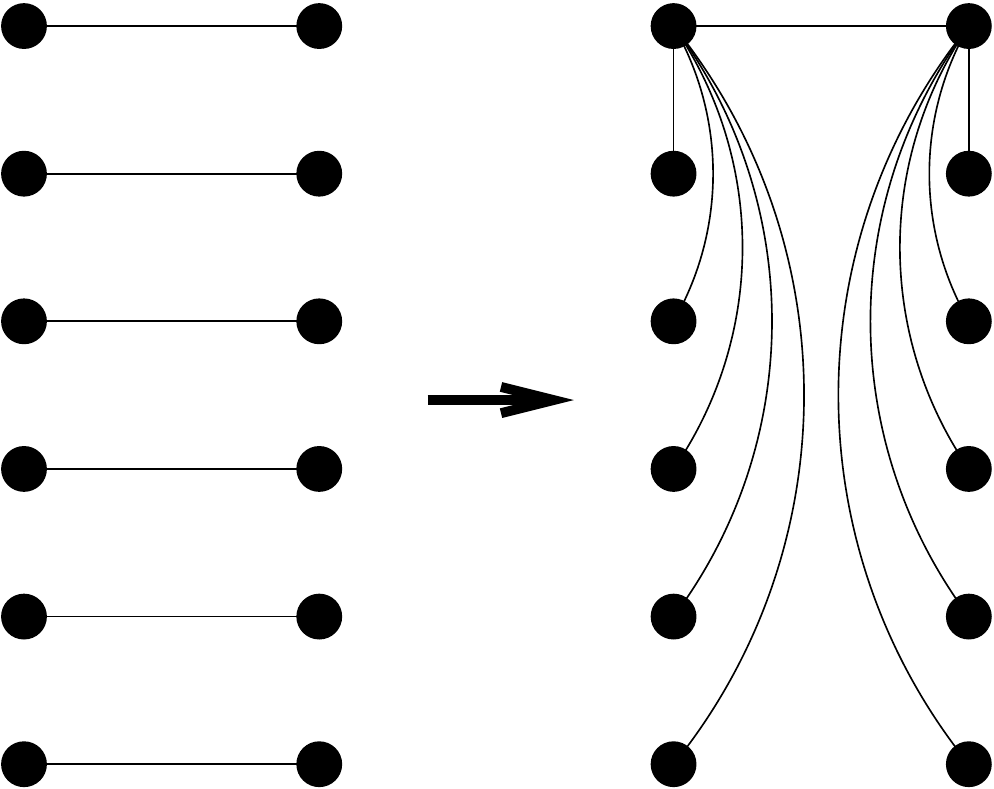}
        }
\put(1.3,6.5){$D$}
\put(6.3,6.5){$D-5$}
\end{picture}
\caption{A switching of colour 12}
\label{f:switching12}
\end{center}
\end{figure}

Suppose that the multigraph $R$ can be formed from $Q$ by
a switching of colour 12.  The number of ways that a switching
of colour 12 can be performed in $Q$ is at least
\[ 2(e_5(Q) + e_6(Q))\,\left([ 2 e_1(Q)]_5 - O(\kmax^2 e_1(Q)^4)\right) \geq 
         (e_5(Q) + e_6(Q))\, M^5,
\]
using Lemma~\ref{mostlyones},
and the number of ways that $R$ can be produced by a switching
of colour 12 is at most $M_5^2$.  Hence we can set
$a_{12}(Q) = (e_5(Q) + e_6(Q)) M^5$ and
$b_{12}(R) = M_5^2$.  This gives
\[ \alpha(Q,R) = \frac{M_5^2}{(e_5(Q)+e_6(Q)) M^5}
                  = O\left(\frac{\kmax^8}{M^3}\right)
                  = o\left(\frac{\kmax^3}{M}\right).
\]

\subsection{Switchings of colour 13}\label{ss:switching13}

Switchings of colour 13 are used to reduce the number of simple
loops until this number is at most $\lceil N_1/2\rceil$.
They are applied to multigraphs $Q$ for which switchings
of colours $1,\ldots, 12$ do not apply and
$\ell_1(Q) > \lceil N_1/2\rceil$. 

Let $Q$ be such a multigraph.  Then a switching of colour 13
is defined by a sequence of distinct vertices $(v_0,v_1,v_2)$
such that there is a simple loop on $v_0$ in $Q$, the
edge $\{v_1,v_2\}$ is a simple edge in $Q$, and the edges
$\{ v_0,v_1\}$, $\{ v_0,v_2\}$ are both absent in $Q$
(that is, they have multiplicity zero).  The switching
removes the simple loop at $v_0$ and the simple edge
$\{ v_1,v_2\}$ and inserts the two simple edges
$\{ v_0,v_1\}$, $\{ v_0,v_2\}$,
as shown in Figure~\ref{f:switching13}.

\begin{figure}[ht]
\begin{center}
\unitlength=0.9cm
\begin{picture}(9,2.1)(0,0)
\put(0,0.0){
    \includegraphics[scale=0.72]{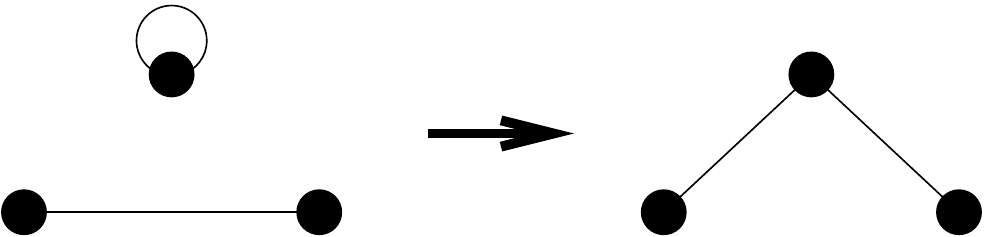}
        }
\end{picture}
\caption{A switching of colour 13}
\label{f:switching13}
\end{center}
\end{figure}

Suppose that the multigraph $R$ can be obtained from $Q$ by performing
a switching of colour 13.  The number of ways that a switching of
colour 13 can be performed in $Q$ is at least 
\[ \ell_1(Q)\left(2e_1(Q) - O(\kmax^2)\right) \geq \nfrac{1}{2}\ell_1(Q) M,\]
using Lemma~\ref{mostlyones},
and the number of ways that a switching of colour 13 can produce $R$
is at most $M_2$.  Therefore we can define
$a_{13}(Q) = \nfrac{1}{2}\ell_1(Q)M$ and $b_{13}(R) = M_2$,
leading to
\[ \alpha(Q,R) = \frac{2\,M_2}{\ell_1(Q) M}
                   \leq \frac{4\, M_2}{N_1 M}
                  \leq \frac{1}{120}, 
\]
using the definition of $N_1$.

\subsection{Switchings of colour 14}\label{ss:switching15}

Switchings of colour 14 reduce the number of non-loop edges
with multiplicity 2 until this number is at most $\lceil N_2/2\rceil$.
This switching is applied to multigraphs $Q$ such that
switchings of colours $1,\ldots, 13$ do not apply and
$e_2(Q) > \lceil N_2/2\rceil$.

Let $Q$ be such a multigraph.  A switching of colour 14 in $Q$
is described by a sequence $(v_0,w_0,v_1,w_1,v_2,w_2)$ of distinct
vertices such that 

\vspace{-\topsep}
\begin{itemize} \itemsep=0pt
\item $\{ v_0,w_0\}$ is an edge of multiplicity 2 in $Q$,
\item $\{ v_1,w_1\}$ and $\{ v_2,w_2\}$ are simple edges in $Q$ 
          (with multiplicity 1),
\item none of the edges $\{ v_0,v_j\}$ or $\{w_0,w_j\}$ are present
in $Q$, for $j=1, 2$ (these are all non-edges, with multiplicity
zero).
\end{itemize}
The switching removes these $3$ edges and replaces them with two
copies of $K_{1,2}$, one on $\{ v_0,v_1,v_2\}$ centred
at $v_0$ and the other on $\{ w_0,w_1,w_2\}$ centred at
$w_0$.
This operation is shown in Figure~\ref{f:switching14}. 

\begin{figure}[ht]
\begin{center}
\unitlength=0.9cm
\begin{picture}(9,3)(0,0)
\put(0,0.0){
    \includegraphics[scale=0.72]{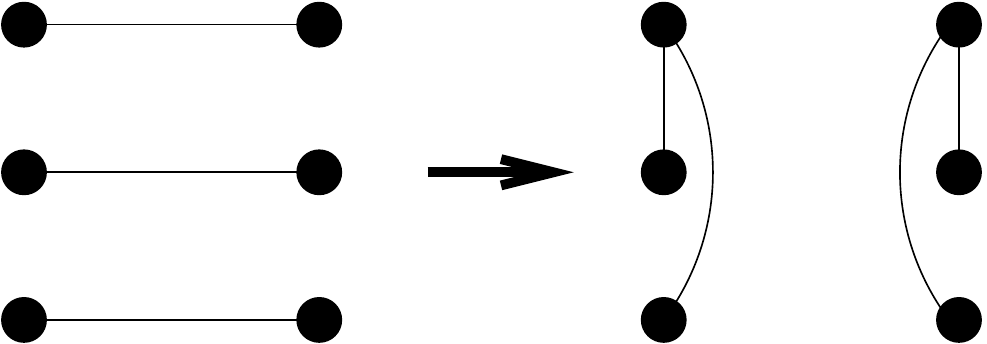}
        }
\put(1.4,2.85){$2$}
\end{picture}
\caption{A switching of colour 14}
\label{f:switching14}
\end{center}
\end{figure}

Suppose that the multigraph $R$ can be produced from $Q$
using a switching of colour 14.  The number of ways to perform
a switching of colour 14 in $Q$ is at least
\[ 2e_2(Q) \left( [2e_1(Q)]_2 - O(\kmax^2 e_1(Q))\right) \geq  e_2(Q) M^2,\]
using Lemma~\ref{mostlyones},
and the number of ways that $R$ can be produced using a
switching of colour 14 is at most $M_2^2$.  Therefore we can
take $a_{14}(Q) =  e_2(Q) M^2$ and $b_{14}(R) = M_2^2$,
leading to
\[ \alpha(Q,R) = \frac{M_2^2}{e_2(Q) M^2}
           < \frac{2M_2^2}{N_2\, M^2} \leq  \frac{1}{120},\]
by the definition of $N_2$.

\subsection{Switchings of colour 15}\label{ss:switching14}

Switchings of colour 15 are used to reduce the number of non-loop
edges with multiplicity 3 until this number is at most
$\lceil N_3/2\rceil$.
They are applied to multigraphs $Q$ such that
switchings of colours $1,\ldots, 14$ do not apply and
$e_3(Q) > \lceil N_3/2\rceil$. 

Let $Q$ be such a multigraph.  Then a switching of colour 15
is defined by a sequence of distinct vertices 
$(v_0,w_0,v_1,w_1,v_2,w_2,v_3,w_3)$
such that

\vspace{-\topsep}
\begin{itemize} \itemsep=0pt
\item $\{ v_0,w_0\}$ has multiplicity 3 in $Q$,
\item each of $\{ v_j,w_j\}$ is a simple edge in $Q$, for $j=1,2,3$,
\item the edges $\{ v_0,v_j\}$ and $\{ w_0,w_j\}$ are absent
in $Q$ for $j=1,2,3$.
\end{itemize}
The switching removes the edges $\{ v_j,w_j\}$ for $j=0,1,2,3$
(setting the multiplicity of each to zero) and inserts the simple
edges $\{ v_0,v_j\}$, $\{ w_0,w_j\}$ for $j=1,2,3$,
as illustrated in Figure~\ref{f:switching11} for the
case that $D=3$.

Suppose that the multigraph $R$ can be produced by performing
a switching of colour 15 from~$Q$.  The number of ways to perform a
switching of colour 15 in $Q$ is at least
\[ 2 e_3(Q)\,\left( [2 e_1(Q)]_3 - O(\kmax^2 e_1(Q)^2)\right) \geq  e_3(Q) M^3,\]
using Lemma~\ref{mostlyones},
while the number of ways that $R$ could be produced using a
switching of colour 15 is at most $M_3^2$.  Therefore we
can take $a_{15}(Q) = e_3(Q) M^3$ and
$b_{15}(R) = M_3^2$, leading to
\[ \alpha(Q,R) = 
   \frac{M_3^2}{ e_3(Q) M^3}
   \leq \frac{2\,M_3^2}{N_3\, M^3}
                  \leq  \frac{1}{120},\]
by the definition of $N_3$.

\subsection{Analysis of the switchings}\label{ss:combine}

We now explain how to apply Theorem~\ref{switchings} to analyse
these switchings on $\G_0$.  (See the statement of Theorem~\ref{switchings}
for the necessary notation.)

We now define a directed graph $\varGamma$ with vertex set 
$V(\varGamma) = \{ v_Q : Q\in \G_0\}$,
where $v_Q$ is associated with the set $\class(v_Q) = \{ Q\}$ 
containing one object.  These sets are certainly disjoint.
By a slight abuse of notation we will identify $v_Q$ and $Q$ from
now on, and write $Q$ for both.   The edge set of $\varGamma$ is defined as
follows: there is an edge $QR$ in $\varGamma$
with colour $c\in\{ 1,\ldots, 15\}$ if and only if $R$ can be obtained 
from $Q$ using a switching of colour~$c$.  
Since a switching is used only where no switching with a lower-labelled
colour applies, there is at most one edge from $Q$ to $R$ in $\varGamma$. 
Hence the endvertices of an edge uniquely determine the colour of the edge.
We take $\lambda_c(v)=\tfrac1{15}$ for all $c,v$, and so
for each edge $QR$ in $\varGamma$ we have
\[ \hat\alpha(QR) = \frac{\alpha(QR)}{\lambda_{c(QR)}(R)}
                  =15\, \alpha(QR).
\]
Let $Y, Z$ be as defined in (\ref{Ydef}), (\ref{Zdef}).

\begin{lemma}
With notation as established above, we have
\[ G(\kvec,\J,\J^*) = \(1 + O(\kmax^3/M)\)\, |\G_0 - Y|.\]
\label{switchinganalysis}
\end{lemma}

\begin{proof}
It follows from the analysis of the previous sections that
$\alpha(QR) \leq \nfrac{1}{30}$, and therefore
$\hat\alpha(QR) \leq \nfrac{1}{2}$, for all edges $QR$ in $\varGamma$.
Moreover, we have provided at least one switching which can be performed
from $Q$, for each 
graph $Q$ in $\G_0-Z$.
Hence $Y$ and $Z$ satisfy the requirements of
Theorem~\ref{switchings} and we can conclude that
\[ |Y| \leq \frac{\hat{\alpha}(Y,Z)}{1-\hat\alpha(Y,Y)} \, |Z| \leq
\frac{\hat{\alpha}(Y,Z)}{1-\hat\alpha(Y,Y)} \, |\G_0|,\]
where
\[ \hat\alpha(W,W') = \max_{P\in\mathcal{P}(W,W')} \hat\alpha(P)\]
for all $W,W'\subseteq \G_0$.

Since $\mathcal{P}(Y,Y)$ by definition has only non-trivial paths,
and $\hat\alpha(QR) \leq \nfrac{1}{2}$ for all edges in~$\varGamma$,
we know that $\hat\alpha(Y,Y) \leq \nfrac{1}{2}$, and therefore
\[ |Y| \leq 2\,\hat\alpha(Y,Z)\, |\G_0|. \]

Now let $P$ be a path in $\varGamma$ from some $Q\in Y$ to some element 
of $Z$, such that all internal vertices of $P$ belong to  $\G_0-(Y\cup Z)$.
Let $QR$ be the first edge in $P$.  Then $QR$ cannot have colour
in $\{1,\ldots, 8\}$, since these switchings only produce graphs in~$Y$.  
If $QR$ is coloured with a colour in $\{ 9, 10, 11, 12\}$ then
$\hat\alpha(QR)=O(\kmax^3/M)$,
so $\hat\alpha(P)=O(\kmax^3/M)$.

The remaining possibility is that $QR$ has colour $12+r$ for
some $r\in\{1,2,3\}$, in which case $P$ must contain at least
$\lfloor N_r/2\rfloor$ edges of colour $12+r$.  Our analysis
showed that $\alpha(e) \leq \nfrac{1}{120}$ and hence that
$\hat\alpha(e) \leq \nfrac{1}{8}$, for all such edges~$e$.
By definition $\lfloor N_r/2\rfloor\ge\nfrac12(\log M-1)$ for 
$r\in\{1,2,3\}$.
Therefore,
\[ \hat\alpha(P) 
  \leq  8^{-(\log M - 1)/2}    = O(M^{-1}), \]
which implies that $\hat\alpha(Y,Z) = O(\kmax^3/M)$.  We conclude that
\[ |\G_0 - Y| = \(1 + O(\kmax^3/M)\) \card{\G_0}. \]
But $\G_0 - Y \subseteq \G(\kvec,\J,\J^*)\subseteq \G_0$, and hence
\[ G(\kvec,\J,\J^*) = \(1 + O(\kmax^3/M)\) \card{\G_0-Y}, \]
completing the proof.
\end{proof}

It remains to obtain an asymptotic expression for $\G_0 - Y$,
which we do in the next section.

\section{From pairings to multigraphs}\label{s:pairings}

In this section we work in the \emph{pairing model} (also
called \emph{configuration model}), which we now describe.
This model is standard for working with random graphs of fixed
degrees: see for example~\cite{McKay84}.
Consider a set of $M$ \emph{points} partitioned into 
\emph{cells}
$c_1,\ldots, c_n$, where cell $c_i$ contains $k_i$ points 
for $i=1,2,\ldots, n$.
Take a partition $P$ (called a \emph{pairing}) of the $M$ points
into $M/2$ \emph{pairs} with each pair having the form $\{ y,z\}$
where $y\in c_i$ and $z\in c_j$ for some $i,j$.  The set of all
such pairings, of which there are $M!/(M/2)!\,2^{M/2}$, will
be denoted by $\pairings(\kvec)$.

A \emph{loop} is a pair whose two points lie in the same cell,
while a \emph{link} is a pair involving two distinct cells.
Two pairs are \emph{parallel} if they involve the same cells.
A \emph{parallel class} is  a maximal set of mutually parallel pairs.
The \emph{multiplicity} of a parallel class (and of the pairs in
the class) is the cardinality of the class.  As important special
cases, a \emph{simple pair} is a parallel class of multiplicity one,
a \emph{double pair} is a parallel class of multiplicity two and
a \emph{triple pair} is a parallel class of multiplicity three.

Each pairing gives rise to a multigraph in $\G(\kvec)$ by
replacing each cell by a vertex, and letting the multiplicity
of the edge $\{v,w\}$ equal the multiplicity of the parallel class
between the corresponding cells.  

Let $\pairings_{\ell,d,t}$ be the set of all pairings in $\pairings(\kvec)$
with exactly $\ell$ simple loops, exactly $d$ double pairs and exactly
$t$ triple pairs, but with no loops of multiplicity greater than one
and no links of multiplicity greater than three.  
If $G\in \G_0 - Y$ can be formed from a pairing $P\in\pairings_{\ell,d,t}$,
then exactly 
\[ 2^{-(\ell+d)}\, 6^{-t}\, \prod_{i=1}^n k_i! \]
pairings in $\pairings(\kvec)$ give rise to $G$.
Now defining
\[ w(\ell,d,t) = 2^{\ell+d}\, 6^t\, \ac{\ell,d,t},\]
we can write
\begin{equation}
\label{counting-pairings}
|\G_0 - Y| = \biggl(\,\prod_{i=1}^n k_i!\biggr)^{\!\!-1}
        \, \sum_{\ell=0}^{N_1}\,\sum_{d=0}^{N_2}\,\sum_{t=0}^{N_3}
                \,w(\ell,d,t).
\end{equation}
Hence it suffices to obtain an asymptotic expression for the above sum.

We will need the following two summation lemmas adapted from~\cite{GMW}:

\begin{lemma}[{\cite[Corollary 4.3]{GMW}}]\label{sumcor}
Let $0\leq A_1\leq A_2$ and $B_1\leq B_2$ be real numbers.
Suppose that there exist integers $N$, $K$ with
$N\geq 2$ and  $0\leq K\leq N$, and a real number
$c> 2e$ such that $Ac<N-K+1$ and $\abs{BN}<1$
for all $A\in[A_1,A_2]$ and $B\in[B_1,B_2]$. 
Further suppose that there are real numbers $\delta_i$, for $1\le i\le N$,
and\/ $\gamma_i\ge 0$,  for $0\le i\le K$, such that\/
$\sum_{j=1}^i \abs{\delta_j}\le \sum_{j=0}^K\gamma_j\ff ij<\tfrac15$
for\/ $1\le i\le N$.\\[1ex]
Given $A(1),\ldots,A(N)\in[A_1,A_2]$ and
$B(1),\ldots,B(N)\in[B_1,B_2]$,
define $n_0,n_1,\ldots,n_N$ by $n_0=1$ and
$$n_i = \frac 1i A(i)\(1 - (i-1)B(i)\)\(1+\delta_i) \,n_{i-1}$$
for\/ $1\le i\le N$. 
Then  $$\varSigma_1 \le \sum_{i=0}^N n_i\le \varSigma_2,$$
where
\begin{align*}
 \varSigma_1 &= \exp\Bigl( A_1 - \tfrac12 A_1^2B_2 
     - 4 \sum_{j=0}^K\gamma_j(3A_1)^j\Bigr) - \tfrac14 (2e/c)^N,\\
 \varSigma_2 &= \exp\Bigl( A_2 - \tfrac12 A_2^2B_1 
 + \tfrac12 A_2^3B_1^2
     + 4 \sum_{j=0}^K\gamma_j(3A_2)^j\Bigr) + \tfrac14 (2e/c)^N.
     \quad\qedsymbol
\end{align*}
\end{lemma}

\begin{lemma}[{\cite[Corollary 4.5]{GMW}}]\label{sumcor2}
Let $N\geq 2$ be an integer and, for $1\leq i\leq N$, let real
numbers $A(i)$, $C(i)$ be given such that $A(i)\geq 0$ and
$A(i)-(i-1)C(i) \ge 0$.
Define 
$A_1 = \min_{i=1}^N A(i)$, $A_2 = \max_{i=1}^N A(i)$,
$C_1 = \min_{i=1}^N C(i)$ and $C_2=\max_{i=1}^N C(i)$.
Suppose that there exists a real number $\hat{c}$ with 
$0<\hat{c} < \tfrac{1}{3}$ such that 
$\max\{ A/N,\, \abs{C}\} \leq \hat{c}$ 
for all $A\in [A_1,A_2]$,  $C\in [C_1,C_2]$.
Define $n_0,\ldots ,n_N$ by $n_0=1$ and
\[ n_i = \frac{1}{i}\(A(i)-(i-1)C(i)\)\, n_{i-1} \]
for $1\leq i\leq N$.  Then
\[ \varSigma_1 \leq \sum_{i=0}^N n_i\leq \varSigma_2, \]
where
\begin{align*}
 \varSigma_1 &= \exp\( A_1 - \tfrac{1}{2} A_1 C_2 \)
               - (2e\hat{c})^N,\\
 \varSigma_2 &= \exp\( A_2 - \tfrac{1}{2} A_2 C_1 +
              \tfrac12 A_2 C_1^2 \) + (2e\hat{c})^N.
     \quad\qedsymbol
\end{align*}
\end{lemma}

In the proofs of Lemmas 4.3--4.5, we will use several results
which were proved in~\cite{MW91} (specifically, Lemmas
4.1--4.4 of that paper and some details of their proofs).
That paper uses values of $N_1,N_2,N_3$ that differ
from ours, but only by bounded factors, and examination
of the proofs in~\cite{MW91} shows that the
results we wish to apply remain valid when our values of
$N_1,N_2,N_3$ are used.

\medskip

First we perform a summation over the number of edges of multiplicity
three.

\begin{lemma}\label{tripcount}
Uniformly for $0\le d\le N_2$ and $0\le \ell\le N_1$, we have
\[  \sum_{t=0}^{N_3} w(\ell,d,t) = w(\ell,d,0)\,
    \exp\biggl( \frac{M_3^2}{2M^3}
      + O\(\kmax^3/M\)\biggr). \]
\end{lemma}

\begin{proof}
We will apply the proof of~\cite[Lemma~4.1]{MW91}.

Let $t'$ be the first value of $t\leq N_3$ for which $\pairings_{\ell,d,t}=\emptyset$,
or $t'=N_3+1$ if there is no such value.
In~\cite[Lemma~4.1]{MW91}, a switching is described that converts any pairing
in $\pairings_{\ell,d,t}$ to at least one in $\pairings_{\ell,d,t-1}$, for $1\le t\le N_3$,
so we know that $w(\ell,d,t)=0$ for $t'\le t\le N_3$.
In particular, the present lemma is true when $w(\ell,d,0)=0$,
so we assume that $t'\ge 1$.

Noting that 
\[ \frac{w(\ell,d,t)}{w(\ell,d,t-1)}
  = \frac{6\,\ac{\ell,d,t}}{\ac{\ell,d,t-1}} \]
when the denominators are nonzero, the calculation in~\cite[Lemma~4.1]{MW91}
shows that there is some uniformly bounded function $\alpha_t=\alpha_t(\ell,d)$ 
such that
\begin{equation}\label{trat}
   \frac{w(\ell,d,t)}{w(\ell,d,0)} 
  = \frac 1t \, \frac{w(\ell,d,t-1)}{w(\ell,d,0)}\, \(A(t) - (t-1)C(t)\)
\end{equation}
for $1\le t\le N_3$, where
\[ A(t)= \frac{M_3^2 - \alpha_t \kmax^2(\kmax^2+\ell+d)M_3}{2M^3},\quad
   C(t) = \frac{\alpha_t \kmax^2 M_3}{2M^3} \]
for $1\le t<t'$ and $A(t)=C(t)=0$ for $t\ge t'$.

Now we can apply Lemma~\ref{sumcor2}.
It is clear that $A(t)-(t-1)C(t)\ge 0$ by~\eqref{trat}.  If $\alpha_t\ge 0$
then $A(t)\ge A(t)-(t-1)C(t)\ge 0$, while if $\alpha_t< 0$ then the 
definition of $A(t)$ makes it evidently nonnegative.  Now define
$A_1,A_2, C_1,C_2$ by taking the minimum and maximum of
$A(t)$ and $C(t)$ over $1\leq t\leq N_3$. 
Let $A\in [A_1,A_2]$ and $C\in [C_1,C_2]$, and set
$\hat{c}=\tfrac1{80}$.
Since $A=M_3^2/2M^3+o(1)$ and $C=o(1)$, we have that
$\max\{A/N_3,\abs{C}\}<\hat{c}$ 
for $M$ sufficiently large, by the definition of $N_3$.

Therefore Lemma~\ref{sumcor2} applies, and gives an 
upper bound
\[ \sum_{t=0}^{N_3} \frac{w(\ell,d,t)}{w(\ell,d,0)}
   \le \exp\biggl( \frac{M_3^2}{2M^3}
            + O\(\kmax^4(\kmax^2+\ell+d)/M^2\)\biggr) +
                             O\((e/40)^{N_3}\). \]
Since $\ell + d \leq N_1 + N_2 = O(\kmax^2 + \log M)$
and $(e/40)^{N_3} \leq (e/40)^{\log M} \leq M^{-2}$,
\[  \sum_{t=0}^{N_3} \frac{w(\ell,d,t)}{w(\ell,d,0)}
  \le \exp\biggl( \frac{M_3^2}{2M^3}
      + O\(\kmax^3/M\)\biggr). \]
In the case that $t'=N_3+1$, the lower bound given by 
Lemma~\ref{sumcor2} is the same within the stated
error term, so we are done.

This leaves the case $1\le t'\le N_3$.  
By the counts of the second switching in the proof of
\cite[Lemma~4.1]{MW91}, $\ac{\ell,d,t}=0$ is only possible
if $M_3=O(\kmax^2(\kmax^2+\ell+d+t))$.  If this happens
for $t\le N_3$ we have 
\[ M_3=O(\kmax^2(\kmax^2+N_1+N_2+N_3)) 
= O(\kmax^2(\kmax^2 + \log M)).\]
However, this implies that $M_3^2/M^3=O(\kmax^3/M)$ so
the upper bound matches the trivial lower bound 1 within the
error term.
This completes the proof.
\end{proof}

Next we perform a summation over the number of simple loops.

\begin{lemma}\label{loopcount}
Uniformly for $0\le d\le N_2$, we have
\[  \sum_{\ell=0}^{N_1}\, w(\ell,d,0) =
    w(0,d,0) \,\exp\biggl( \frac{M_2}{M}
      + O\biggl(\frac{\kmax d}{M}+\frac{\kmax^3}{M}\biggr)\biggr). \]
\end{lemma}

\begin{proof}
Let $\ell'$ be the first value of $\ell\le N_1$ for which
$\pairings_{\ell,d,0}=\emptyset$,
or $\ell'=N_1+1$ if there is no such value.
In the proof of~\cite[Lemma~4.2]{MW91}, a switching is described
that converts any pairing
in $\pairings_{\ell,d,0}$ to at least one in $\pairings_{\ell-1,d,0}$,
for $1\le \ell\le N_1$,
so we know that $w(\ell,d,0)=0$ for $\ell'\le \ell\le N_1$.
In particular, the present lemma is true when $w(0,d,0)=0$,
so we assume that $\ell'\ge 1$.

Noting that 
\[ \frac{w(\ell,d,0)}{w(\ell-1,d,0)}
  = \frac{2\,\ac{\ell,d,0}}{\ac{\ell-1,d,0}} \]
when the denominators are nonzero, the calculation in~\cite[Lemma~4.2]{MW91}
shows that there is some uniformly bounded function $\beta_\ell=\beta_\ell(d)$ 
such that
\[
   \frac{w(\ell,d,0)}{w(0,d,0)} 
  = \frac 1\ell \, \frac{w(\ell-1,d,0)}{w(0,d,0)}
     \,\(A(\ell) - (\ell-1)C(\ell)\)
\]
for $1\le t\le N_1$, where
\[ A(\ell)= \frac{M_2 - \beta_\ell(\kmax^3+\kmax d)}{M},\quad
   C(\ell) = \frac{\beta_\ell \kmax^2}{M}, \]
for $1\le \ell<\ell'$ and $A(\ell)=C(\ell)=0$ for $\ell\ge \ell'$.

We can now complete the proof using $\hat c=\frac{1}{80}$ and
following the argument used in the previous lemma.
The treatment of the lower bound when $\ell'\le N_1$ needs
some additional care.  From the analysis of the second switching
used in~\cite[Lemma~4.2]{MW91}, $\ell'\le N_1$ can only
happen if $M_2=O(\kmax d + \kmax^2\ell')$.
For $\ell'=1$, the trivial lower bound of 1 matches
the upper bound within the stated error terms.
If $2\le\ell'\le N_1$ then the first two terms of the summation
give the lower bound
\[
w(0,d,0)(1 + A(1)) 
    = w(0,d,0)\exp\left(\frac{M_2}{M} + O\left(\frac{\kmax d}{M} + \frac{\kmax^3}{M}
                 \right)\right)
\]
which matches the upper bound within the stated error terms.
In either case, the proof is complete.
\end{proof}

Next we perform a summation over the number of edges of multiplicity
two.  The exponential factor in the summand corresponds to the
error term from Lemma~\ref{loopcount} which depends on $d$.

\begin{lemma}\label{doublecount}
For any constant $\rho$ have
\[  \sum_{d=0}^{N_2} w(0,d,0)
   \exp\biggl(\frac{\rho \kmax d}{M}\biggr) = w(0,0,0) \,
    \exp\biggl( \frac{M_2^2}{2M^2} - \frac{M_3^2}{2M^3}
      + O\(\kmax^3/M\)\biggr). \]
\end{lemma}

\begin{proof}
Let $d'$ be the first value of $d\le N_2$ for which $\pairings_{0,d,0}=\emptyset$,
or $d'=N_2+1$ if no such value of $d$ exists.
As in the previous two lemmas, the switchings described in~\cite{MW91}
(see also~\cite[Lemma 4]{MW90b})
show that $w(0,d,0)=0$ for $d'\le d\le N_2$.
They also show that $d'\le N_2$ is only possible if
$M_2=O(\kmax^3+\kmax d)$.
In particular, the lemma is true if $w(0,0,0)=0$, so we assume
that $d'\ge 1$.

We divide the proof into two cases, following the division used
in~\cite[Lemmas 4.3--4.4]{MW91}.
For $0\le d\le N_2$, define
\[
  m_d = w(0,d,0) \, \exp\biggl(\frac{\rho \kmax d}{M}\,\biggr).
\]

First suppose that $M_2\le M$. From~\cite[Lemma~4.3]{MW91} we
know that for $1\le d<d'$,
\[
    \frac{m_d}{m_{d-1}} =
   \frac{2\, \exp(\rho\kmax/M)\, \ac{0,d,0}}{\ac{0,d-1,0}} =
      \frac{M_2^2}{2dM^2} + O\biggl( \frac{\kmax(\kmax^2+d)M_2}{dM^2}\biggr).
\]
The current lemma follows from this, arguing as in the proofs of the previous
two lemmas.  (Note that the term $M_3^2/2M^3$ in the answer is absorbed into the error term, by the assumption that $M_2\leq M$.)
In the case that $d'\le N_2$, the condition $M_2=O(\kmax^3+\kmax d)$
implies that 1 is a sufficient lower bound.

\medskip

Now suppose that $M_2>M$.
In this case $M_2=O(\kmax^3+\kmax d)$ is not possible
for $d\le N_2$, so $d'=N_2+1$ and the series contains only positive terms.
By~\cite[Lemma~4.4]{MW91} we have
for $1\le d\le N_2$,
\[
     m_d = m_{d-1} \,
        \frac{A(d)}{d}\(1-(d-1)B(d)\)\,(1+\delta_d), 
\]
where
\begin{align*}
  A(d) &= \frac{M_2^2}{2M^2} + \frac{2M_2^2M_3}{M^4}
             - \frac{M_3^2}{2M^3} - \frac{M_2^4}{M^5} + O(\kmax^3/M), \\[0.4ex]
  B(d) &= B = -\frac{8}{M} + \frac{16M_3}{M_2^2}, \\[0.4ex]
  \delta_d &= O\((d-1)/M_2\) \,\,\,
     \text{ uniformly for $1\leq d\leq N_2$}.
\end{align*}
(Note that $A(d)\exp(\rho\kmax/M)$ equals $A(d)$ within the precision afforded
by the error term, since $\kmax A(d)/M = O(\kmax M_2^2/M^3) = O(\kmax^3/M)$.)

Now we will apply Lemma~\ref{sumcor} to~$\sum_d m_d$.
Let $A_1=\min_d A(d)$ and $A_2=\max_d A(d)$, with the minimum and
maximum taken over $1\le d\le N_2$, and let $B_1=B_2=B$.
Define $c=80$ and $K=2$.  Then the conditions of Lemma~\ref{sumcor}
apply with $\gamma_0=\gamma_1=0$ and some value of $\gamma_2$
which satisfies $\gamma_2=O(1/M_2)$.
Application of that lemma now gives the desired result.
\end{proof}

Recall the definition of $y_1$, $x_2$, $x_3\in \{0,1\}$ 
given just before the statement of Theorem~\ref{main}.
By combining the last three summations we obtain the following.

\begin{lemma}\label{allcount}
  As $M\to\infty$ and $\kmax^3=o(M)$, we have
  \begin{align*}
      \sum_{d=0}^{x_2N_2}\, \sum_{\ell=0}^{y_1N_1}\, &\sum_{t=0}^{x_3N_3}\,
        w(\ell,d,t)\\
         &= w(0,0,0) \,
      \exp\biggl(  y_1\frac{M_2}{M} +
        x_2 \frac{M_2^2}{2M^2}  + (x_3-x_2)\frac{M_3^2}{2M^3}
      + O\(\kmax^3/M\)\biggr).
    \end{align*}
\end{lemma}

\begin{proof}
  From Lemmas~\ref{tripcount} and \ref{loopcount}, we have
  \begin{equation}\label{ltsum}
   \sum_{\ell=0}^{y_1N_1}\,\sum_{t=0}^{x_3N_3}
    \,w(\ell,d,t)
    =  w(0,d,0)\, \exp\biggl(
      y_1\frac{M_2}{M} + x_3\frac{M_3^2}{2M^3}
      + O\biggl(\frac{\kmax d}{M}+\frac{\kmax^3}{M}\biggr)\biggr)
  \end{equation}
  uniformly over $d$.

  If $x_2=0$, we are finished.  If $x_2=1$, choose two constants
  $\rho_1,\rho_2$ such that the actual value of the error term
  $O(\kmax d/M)$ in~\eqref{ltsum}
  lies in $[\rho_1\kmax d/M,\rho_2\kmax d/M]$ for all 
$d\in \{ 1,\ldots, N_2\}$.
  Applying Lemma~\ref{doublecount} with $\rho\in\{\rho_1,\rho_2\}$
  and noting that the result does not depend on~$\rho$ within the
  precision given by the error term, we are done.~%
\end{proof}

Finally we can prove our main result.

\begin{proof}[Proof of Theorem~\ref{main}]
From Lemma~\ref{switchinganalysis}
and Lemma~\ref{allcount} together with
(\ref{counting-pairings}),
we obtain
\begin{align*}
G(\kvec,\J,\J^*) &= \(1 + O(\kmax^3/M)\)\, 
               |\G_0 - Y|\\
  &= \frac{1+O\(\kmax^3/M\)}{k_1!\cdots k_n!}\,\,
       \sum_{d=0}^{x_2N_2}\, \sum_{\ell=0}^{y_1N_2}\,
          \sum_{t=0}^{x_3N_3} \, w(\ell,d,t)\\
  &=  \frac{w(0,0,0)}{k_1!\cdots k_n!}\, 
      \exp\biggl(  y_1\frac{M_2}{M} +
        x_2 \frac{M_2^2}{2M^2}  + (x_3-x_2)\frac{M_3^2}{2M^3}
      + O\(\kmax^3/M\)\biggr).   
\end{align*}
Now $w(0,0,0) = \ac{0,0,0}$, and it follows from~\cite[Lemma 5.1]{MW91}
that
\begin{align*}
\ac{0,0,0} &= \frac{M!}{(M/2)!\,2^{M/2}}\,
 \exp\left(\!-\frac{M_2}{2M} - \frac{M_2^2}{4M^2} - 
   \frac{M_2^2M_3}{2M^4} + \frac{M_2^4}{4M^5} + \frac{M_3^2}{6M^3}
     + O(\kmax^3/M)\right).
\end{align*}
Combining these two expressions completes the proof.
\end{proof}

\section{Comparison to a na\"\i ve model}\label{s:naive}

Let $p\in (0,1)$ be a probability which we will define later.
Abusing notation slightly, we define the functions
\[ J(z) = \sum_{j\in J} p^j z^j,\quad
 J^*(z) = \sum_{j\in J^*} p^j z^j
\]
and the probability generating functions 
\[ f(z) = \frac{J(z)}{J(1)}, \quad
   g(z) = \frac{J^*(z)}{J^*(1)}.\]
Consider a random symmetric nonnegative integer matrix $A=(a_{ij})$ which
is created as follows:  
\begin{itemize} 
\item All entries on and above the diagonal are independent;
\item For $i=1,\ldots, n$, let
$a_{ii}$ be a randomly chosen element of $J^*$, with\\
$\Pr(a_{ii}=b) = p^b/J^*(1)$ for all $b\in J^*$; 
\item For $1\leq i < j\leq n$ let $a_{ij}$ be a randomly chosen 
element of $J$, with\\ $\Pr(a_{ij}=b) = p^b/J(1)$ 
for all $b\in J$; 
\item Finally, let $a_{ji}=a_{ij}$ for $1\leq i < j\leq n$.
\end{itemize}
We write $\Pr(\cdot)$ to mean probabilities generated by the
above procedure.
If $A_0$ is a fixed matrix which corresponds to an element of 
$\mathcal{G}(\kvec,J,J^*)$ then the probability that $A_0$ is
produced by the above procedure is
\begin{equation}
  \Pr(A_0) = \frac{p^{M/2}}{J(1)^{\binom{n}{2}} J^*(1)^n}.
\label{fred}
\end{equation}
This follows as the probability of $A_0$ depends only on the
sum of the diagonal and above-diagonal
entries (that is, on the number of edges of the corresponding
multigraph).  

For $i=1,\ldots, n$, let $\mathcal{E}_i$ denote the event that the
sum of row $i$ equals $k_i$,
recalling that diagonal entries are weighted by a factor of 2
as in~\eqref{rowsum}.
Now
\[ G(\kvec,J,J^*) =  \frac{\Pr(\mathcal{E}_1\wedge \mathcal{E}_2
    \wedge\cdots\wedge \mathcal{E}_n)}{\Pr(A_0)}
\]
since $\Pr(\cdot)$ is uniform on
(the set of matrices corresponding to) $\mathcal{G}(\kvec,J,J^*)$.
By (incorrectly) assuming that the events 
$\mathcal{E}_1,\ldots, \mathcal{E}_n$ are independent,
we obtain the following na{\"\i}ve estimate of $G(\kvec,J,J^*)$, 
using (\ref{fred}):
\begin{align}
 \Gp(\kvec,J,J^*) 
  &= \frac{\prod_{i=1}^n \Pr(\mathcal{E}_i)}{\Pr(A_0)} \notag\\
  &= \frac{J(1)^{\binom n2} J^*(1)^n} {p^{M/2}}\,
\prod_{i=1}^n\, [z^{k_i}] \,f(z)^{n-1} g(z^2) \notag \\
 &= p^{-M/2}\, J(1)^{-\binom{n}{2}} 
  \, \prod_{i=1}^n \, [z^{k_i}]\,  J(z)^{n-1} J^*(z^2).\label{Gnaive}
\end{align}

Let $\kbar=M/n$ denote the target average row sum.
The value $p_0$ of $p$ which makes the expected row sum equal to $k$
satisfies
\begin{equation}
\label{pdef}
 p_0 = \frac{\kbar}{n} + \frac{(1-2x_2)\kbar^2}{n^2} + O(k/n^2).
\end{equation}
We write $\Gnaive(\kvec,J,J^*)$ for $G_{p_0}(\kvec,J,J^*)$
from now on.  (As we shall see, the exact value of the
$O(k/n^2)$ term does not affect Theorem~\ref{reformulate} below,
within the stated error bound, even though some of the intermediate
formulae we give in the proof are affected.)

We now compare the expression given in  Theorem~\ref{main} with the na\"\i ve
estimate $\Gnaive(\kvec,J,J^*)$.  For future applications, it will be
convenient to express the result
in terms of the scaled central moments $\mu_2$, $\mu_3$ defined by
\[ \mu_r = \frac{1}{M}\, \sum_{i=1}^n (k_i-\kbar)^r\]
for $r=2,3$.

\begin{theorem}
\label{reformulate}
\begin{align}
   G(\kvec,&J,J^*)
   = \sqrt2\, 
            \Gnaive(\kvec,J,J^*) \exp\biggl(
\dfrac14(1-\mu_2)\,\(1+2x_2 +  \mu_2(1-2x_2)\)  \notag \\
  & \hspace*{19mm}  + \frac{6\mu_2\mu_3(x_3-x_2) - \mu_2^3}{2n}
 + \frac{3\mu_2^2(\mu_2^2 - 2\mu_3) + 2\mu_3^2(3x_3-3x_2+1)}{12M} \notag \\
  & \hspace*{19mm}
  + \frac{\mu_2^2 M}{2n^2}(9x_3-9x_2-1) +O(\kmax^3/M)\biggr).
\end{align}
\end{theorem}

\begin{proof}  
Suppose that
\[ p = \frac{\kbar}{n} + \frac{(1-2x_2)\kbar^2}{n^2} + \frac{ck}{n^2},
\]
where $c = c(n) = O(1)$.
By direct computation we find that
\begin{align}
 p^{M/2} &= \biggl(\frac{k}{n}\biggr)^{\!M/2} 
   \exp\biggl(  \dfrac12 \kbar^2(1-2x_2) - \frac{\kbar^3}{4n} 
   + \frac{ck}{2}
     +   O(\kmax^3/M)\biggr) 
    \label{pexp}
\end{align}
     and
\begin{align}
 J(1)^{\binom n2} &= \exp\biggl( \binom{n}{2} p - 
   \frac{(1-2x_2)p^2 n^2}{4} + \frac{(1-3x_2+3x_3)p^3n^2}{6}
                 + O(\kmax^3/M)\biggr)\notag\\
          &= \exp\biggl( \dfrac12 M - \dfrac14 \kbar(2x_2\kbar + 2 -\kbar)
  + \frac{ck}{2}
     - \frac{\kbar^3(2+3x_2-3x_3)}{6n} + O(\kmax^3/M)\biggr). \label{Jexp} 
\end{align}

To calculate the product in (\ref{Gnaive}) we consider two cases, 
depending on the value of~$x_2$.  
First assume that $x_2=0$ (that is, the value 2 is not permitted off the diagonal).
Then
\[
 J(z)^{n-1}J^*(z^2) = (1+pz)^{n-1} H_0(z),
\]
where 
\[         H_0(z) = \biggl( 1+ \frac{\sum_{i\geq 3} x_i p^i z^i}{1+pz} \biggr)^{\!n-1} 
          \Bigl(1 + \sum_{i\geq 1} y_i p^i z^{2i}\Bigr).
\]
Consequently, for any integer $r$ with $0\le r\leq \kmax$ we have
\begin{align}
  [z^r]\, &J(z)^{n-1}J^*(z^2)\,
   = \, \sum_{j=0}^r \binom{n-1}{r-j}p^{r-j} \, [z^j] H_0(z) \notag\\
   &= \binom{n-1}{r} p^r
        + \binom{n-1}{r-2} y_1 p^{r-1} + \binom{n-1}{r-3} \, x_3(n-1)p^r  + 
        \binom{n-1}{r} p^r \varDelta_0,\label{zrJJ}
\end{align}
where
\[
 \varDelta_0 = \sum_{j=4}^r \binom{n-1}{r-j} \binom{n-1}{r}^{\!\!-1} p^{-j}
  \, [z^j] H_0(z).
\]
Now define
\[
H_0^+(z) = \biggl( 1+ \frac{\sum_{i\geq 3} p^i z^i}{1-pz} \,\biggr)^{\!n-1} 
          \biggl(1 + \sum_{i\geq 1} p^i z^{2i}\biggr)
       = \biggl( 1 + \frac{p^3z^3}{(1-pz)^2}\biggr)^{\!n-1} (1-pz^2)^{-1}.
\]
Notice that the coefficients of $H_0^+(z)$, when expanded as a
Taylor series in $z$, are nonnegative and dominate those of $H_0(z)$.
Also notice that $\binom{n-1}{r-j} \binom{n-1}{r}^{\!\!-1}
=O(1)(r/n)^j$ uniformly for $4\le j\le r$, since $r=o(n^{1/2})$.
Therefore, setting $\alpha= r/pn$,
\begin{align*}
  \abs{\varDelta_0} &= O(1) \sum_{j=4}^r \alpha^j\,[z^j]H_0^+(z)\\
       &= O\( H_0^+(\alpha) - 1 - p\alpha^2 - (n-1)p^3\alpha^3\).
\end{align*}
Since $p\alpha=o(1)$, by Taylor's Theorem we can write
\[  \biggl( 1 + \frac{p^3\alpha^3}{(1-p\alpha)^2}\biggr)^{\!n-1}
   \negthickspace = 1 + (n-1)\frac{p^3\alpha^3}{(1-p\alpha)^2} 
        + O(n^2p^6\alpha^6). \]
Substituting this into the definition of $H_0^+(\alpha)$, we find that
\[ H_0^+(\alpha) - 1 - p\alpha^2 - (n-1)p^3\alpha^3
 = \frac{p^2\alpha^4(1+o(1))}{(1-p\alpha)^2(1-p\alpha^2)}
   = O(p^2\alpha^4), \]
and so
\[ \varDelta_0 = O(r\kmax^3/M^2). \]
Applying this bound to~\eqref{zrJJ}, we have
\begin{align}
 [z^r]&\,J(z)^{n-1}J^*(z^2) \notag\\
     &= \binom{n-1}{r} p^r \exp\biggl( \frac{y_1r(r-1)}{kn}
                 + \frac{x_3 r^3}{n^2}  + O(r\kmax^3/M^2) \biggr)\notag \\
     &= \frac{n^r p^r}{r!} \exp\biggl( -\frac{r(r+1)}{2n} + \frac{y_1r(r-1)}{kn}
                 + \frac{ r^3(6x_3-1)}{6n^2} + O(r\kmax^3/M^2) \biggr).
                    \label{x=0case}
\end{align}

Now we assume that $x_2=1$ and perform a similar calculation. 
We have
\[ J(z)^{n-1}\, J^*(z^2) = (1-pz)^{-n+1} H_1(z), \]
where
\[
      H_1(z) =   \biggl((1-pz)\sum_{i\geq 0} x_i p^i z^i\biggr)^{\!n-1}
          \Bigl(1+\sum_{i\geq 1} y_ip^i z^{2i}\Bigr).
\]
We find that for any $r$ with $0\le r\le\kmax$,
\begin{align*}
 [z^r]\,J(z)^{n-1}J^*(z^2) 
  &= \binom{n+r-2}{r} p^r
 + \binom{n+r-4}{r-2} y_1 p^{r-1} \\
 &{\quad}+ \binom{n+r-5}{r-3} (x_3-1) (n-1)\, p^r 
   + \binom{n+r-2}{r}\, p^r  \varDelta_1,
\end{align*}
where
\[ \varDelta_1 = \sum_{j=4}^r \binom{n+r-j-2}{r-j} \binom{n+r-2}{r}^{\!-1} p^{-j}
   [z^r] H_1(z). \]
The coefficients of $H_1(z)$ are dominated by those of $H_1^+(z)$, where
\[
   H_1^+(z) = \biggl( 1 + \sum_{i\ge 3} p^iz^i\biggr)^{\!n-1}
      \biggl(1 + \sum_{i\ge 1}p^iz^{2i}\biggr)
     = \biggl( 1 + \frac{p^3z^3}{1-pz} \biggr)^{n-1} \( 1 - pz^2)^{-1}.
\]
Arguing as before, this implies that $\varDelta_1=O(r\kmax^3/M^2)$,
and so we have
\begin{align}
[z^r]\,J(z)^{n-1}J^*(z^2)
     &= \binom{n+r-2}{r} p^r \exp\biggl( \frac{y_1r(r-1)}{kn}
                 + \frac{(x_3-1) r^3}{n^2}  + O(r\kmax^3/M^2) \biggr)\notag \\
     & {\kern-1.2cm}=\frac{n^r p^r}{r!} \exp\biggl( \frac{r(r-3)}{2n} + \frac{y_1r(r-1)}{kn}
                 + \frac{ r^3(6x_3-7)}{6n^2} + O(r\kmax^3/M^2) \biggr).
                 \label{x=1case} 
\end{align}
Combining the cases \eqref{x=0case} and \eqref{x=1case} we have,
for $i=1,\ldots, n$,
\begin{align*}
[z^{k_i}]\, J(z)^{n-1}J^*(z^2) &= \frac{n^{k_i}\, p^{k_i}}{{k_i}!} \exp\biggl(
   - \biggl(\frac{1 + 2x_2}{2n}  + \frac{y_1}{kn}\biggr)\, k_i -
  \biggl(\frac{1-2x_2}{2n} - \frac{y_1}{kn}\biggr)\, k_i^2 \notag \\
 & \hspace*{3.5cm}  - \frac{1+6x_2-6x_3}{6n^2}\, k_i^3 
 + O(k_i\kmax^3/M^2)\biggr). 
\end{align*}
Multiplying these $n$ equations together gives
\begin{align} 
\prod_{i=1}^n \, [z^{k_i}]&\, J(z)^{n-1}J^*(z^2) \notag\\
&= \frac{n^M p^M}{\prod_{i=1}^n k_i!}\,
  \exp\biggl(  \frac{M_2(2x_2-1) -2M }{2n} + \frac{y_1 M_2}{kn}  \notag \\
  &  \hspace*{40mm}  {}
    - \frac{(6x_2-6x_3+1)M_3}{6n^2} + O(\kmax^3/M)\biggr).
\label{rowexp}
\end{align}

We now substitute (\ref{pexp}), (\ref{Jexp}) and (\ref{rowexp}) into 
(\ref{Gnaive}).
Writing the result in terms
of $\mu_2$, $\mu_3$, $M$ and $n$, using the identities
 \begin{align*}
 M_2 &= M \mu_2  + (k-1)M,\\
 M_3 &= M \mu_3  + 3(k-1)M \, \mu_2 + (k-1)(k-2)M
 \end{align*}
we find that
\begin{align*}
& \Gnaive(\kvec, J,J^*) \\
 &= \frac{1}{\prod_{i=1}^n k_i!} \,
 \left(\frac{M}{e}\right)^{\! M/2} \!\! \exp\biggl( -y_1(1-\mu_2) +
  \frac{(2y_1+ 2x_2(\mu_2-1) - \mu_2)k}{2} +   \frac{(2x_2-1)k^2}{4} 
   \biggr.\\
  & \hspace*{12em}\biggl. {} 
    - \frac{k(6x_2-6x_3+1)(2\mu_3+6 \mu_2 k+k^2)}{12 n} 
    + O(\kmax^3/M)\biggr).
\end{align*}
(At this point we can observe that no remaining
terms depend on $c$, which verifies our earlier claim that the
exact value of the $O(k/n^2)$ term in (\ref{pdef}) does not
affect the statement of this theorem.)
The proof 
is completed by applying Theorem~\ref{main} and using 
Stirling's formula.
\end{proof}

\begin{corollary}
\label{naive-corollary}
Under the conditions of Theorem~\ref{reformulate}, 
if\/ $\mu_2 = O(M^{1/6})$ then
 \begin{align*}   G&(\kvec,J,J^*) \\ 
    &= \sqrt2\, 
         \Gnaive(\kvec,J,J^*) \,
      \exp\Bigl(\, \dfrac14\, (1-\mu_2)\bigl(1 + 2x_2 + \mu_2(1-2x_2)\bigr)
    + O\(\kmax^2/M^{2/3}\)  \Bigr). 
             \end{align*}
If the stronger bound $\mu_2 = O(\kmax^{1/2})$ holds then 
\begin{align*}   G&(\kvec,J,J^*) \\ 
    &= \sqrt2\, 
         \Gnaive(\kvec,J,J^*) \,
      \exp\Bigl( \dfrac14\, (1-\mu_2)\bigl(1 + 2x_2 + \mu_2(1-2x_2)\bigr)
    + O\(\kmax^3/M\)\Bigr). 
             \end{align*}
Finally, if $\kvec=(k,k,\ldots, k)$ is a regular degree sequence
with $kn$ even, then
\begin{equation}\label{magic}
G(\kvec,J,J^*) = \sqrt{2} \, \exp\( \dfrac{1}{4}(1+2x_2)
  +O(k^2/n) \)   \, \Gnaive(\kvec,J,J^*).
\end{equation}
\end{corollary}
\begin{proof}
For the first two statements, 
we just need to check that all additional terms inside
the exponential factor in Theorem~\ref{reformulate} are covered by the
claimed error bounds.
For this it is useful to note that 
$\mu_2\le\kmax$ and $\abs{\mu_3}\le \kmax\mu_2$.
Finally, when $\kvec$ is a regular degree sequence we have 
$\mu_2=\mu_3=0$ and \eqref{magic} follows.
\end{proof}

The constant $\frac14(1+2x_2)$ in~\eqref{magic} was previously
noted for simple sparse regular graphs in~\cite{MW91}, and in~\cite{Loopy}
for simple sparse regular graphs with loops allowed.
Interestingly, with different error terms, the same constant was observed
in~\cite{MW91} for dense simple regular graphs, in~\cite{Loopy} for
dense simple regular graphs with loops allowed, and in~\cite{MM11} for dense
symmetric integer matrices with zero diagonal.
In these three cases, 
we conjectured that~\eqref{magic} holds (with some vanishing error term) for all 
degree sequences
except for the two extreme cases of graphs with no edges and
graphs with all possible edges.  The corresponding conjectures may be
less likely to be true in full generality for all $J,J^*$.

From the first expression of Corollary~\ref{naive-corollary}, 
we see that $G(\kvec,J,J^*)$ is closely approximated by
$\sqrt{2}\, \Gnaive(\kvec,J,J^*)$ whenever $\mu_2$ is close to 1.
It appears that, in the random matrix model described at
the start of this section, $\mu_2$ will be concentrated around 1
with high probability whenever $p$ tends to zero slowly enough to ensure 
that $M\to\infty$ with high probability, but quickly enough to
ensure that $\kmax^3 = o(M)$ with high probability.
If so, this will lead to a model for the degree sequences of sparse
multigraphs analogous to that obtained by McKay and Wormald~\cite{MW97}
for graphs, and by McKay and Skerman~\cite{MS} for bipartite graphs
and directed graphs.
Details will be given in a future paper.

\medskip

The authors are grateful to the referees for their careful reading, which led
to several improvements.

\def\volno#1{\textbf{#1}}
\def\aufont#1{\textsc{#1}}
\def\pagenos#1{#1}


\begin{thebibliography}{99}

\bibitem{BCFMR}
\aufont{J.\,R.~Banavar, F.~Colaiori, A.~Flammini, A.~Maritan and
A.~Rinaldo}, Topology of the fittest transportation network,
\textit{Phys. Rev. Lett.}, \volno{84} (2000), \pagenos{4745--4748}.

\bibitem{BH12}
\aufont{A. Barvinok and J.\,A. Hartigan},
An asymptotic formula for the number of non-negative integer matrices with
prescribed row and column sums,
\textit{Trans. Amer. Math. Soc.}, \volno{364} (2012), \pagenos{4323--4368}.

\bibitem{BC78} 
\aufont{E.\,A.~Bender and E.\,R.~Canfield}, 
The asymptotic number of labeled graphs with given degree sequences, 
\textit{J. Combin. Theory Ser. A}, \volno{24} (1978), \pagenos{296--307}.

\bibitem{GM08}
\aufont{C.~Greenhill and B.\,D.~McKay}, 
Asymptotic enumeration of sparse nonnegative integer matrices with 
specified row and column sums, 
\textit{Adv. in Appl. Math.}, \volno{41} (2008), \pagenos{459--481}.

\bibitem{Loopy}
\aufont{C.~Greenhill and B.\,D.~McKay},
Counting loopy graphs with given degrees,
\textit{Linear Alg. Appl.}, \volno{436} (2012), \pagenos{901--926}.

\bibitem{GFZSTLKS}
\aufont{H.\,H.~Gan, D.~Fera, J.~Zorn, N.~Shiffeldrim, M.~Tang, U.~Laserson,
N.~Kim and T.~Schlick},
RAG: RNA-As-Graphs database - concepts, analysis, and features,
\textit{Bioinformatics}, \volno{20} (2004), \pagenos{1285--1291}.

\bibitem{godehardt}
\aufont{E.\,A.\,J.~Godehardt},
Probability models for random multigraphs with applications in
cluster analysis,
\textit{Ann. Discrete Math.}, \volno{55} (1993), \pagenos{93--108}.
 
\bibitem{GMW}
\aufont{C.~Greenhill, B.\,D.~McKay and X.~Wang},
Asymptotic enumeration of sparse 0-1 matrices with irregular row and column
sums, \textit{J. Combin. Theory Ser. A}, \volno{113} (2006),
\pagenos{291--324}.

\bibitem{HM}
\aufont{M. Hasheminezhad and B.\,D.~McKay},
Combinatorial estimates by the switching method,
\textit{Contemp. Math.}, \volno{531} (2010), \pagenos{209--221}.

\bibitem{KKRR}
\aufont{D.~Knisley, J.~Knisley, C.~Ross and  A.~Rockney},
Classifying Multigraph models of secondary RNA structure
using graph-theoretic descriptors,
\textit{ISRN Bioinformatics}, \volno{2012} (2012), Article ID 157135.

\bibitem{McKay84}
\aufont{B.\,D.~McKay},
\textit{Asymptotics for 0-1 matrices with prescribed line sums},
in Enumeration and Design, S.\,A.~Vanstone and D.\,M.~Jackson,
eds., Academic Press,
Toronto, 1984, \pagenos{225--238}.

\bibitem{ranX}
\aufont{B.\,D.~McKay},
Subgraphs of dense random graphs with specified degrees,
 \textit{Combin. Probab. Comput.}, \volno{20} (2011), \pagenos{413--433}.
  
\bibitem{MM11}
\aufont{B.\,D.~McKay and J.\,C.~McLeod},
Asymptotic enumeration of symmetric integer matrices with uniform
row sums,
\textit{J. Aust. Math. Soc.}, \volno{92} (2012), \pagenos{367--384}.

\bibitem{MS}
\aufont{B.\,D.~McKay and F.~Skerman},
Degree sequences of random digraphs and bipartite graphs,
submitted (2012).  \texttt{http://arxiv.org/abs/1302.2446}

\bibitem{MW90}
\aufont{B.\,D.~McKay and N.\,C.~Wormald},
Asymptotic enumeration by degree sequence of graphs of high degree,
\textit{European J. Combin.}, \volno{11} (1990), \pagenos{565--580}.

\bibitem{MW90b}
\aufont{B.\,D.~McKay and N.\,C.~Wormald},
Uniform generation of random regular graphs of moderate degree,
\textit{J.~Algorithms}, \volno{11} (1990), \pagenos{52--67}.

\bibitem{MW91}
\aufont{B.\,D.~McKay and N.\,C.~Wormald}, 
Asymptotic enumeration by degree sequence of graphs with degrees $o(n^{1/2})$,
\textit{Combinatorica}, \volno{11} (1991), \pagenos{369--382}.

\bibitem{MW97}
\aufont{B.\,D.~McKay and N.\,C. Wormald},
The degree sequence of a random graph. I. The
models, \textit{Random Structures Algorithms},
\volno{11} (1997), \pagenos{97--117}.

\bibitem{Read} 
\aufont{R.\,C.~Read},
\textit{Some enumeration problems in graph theory},
 Doctoral Thesis, University of London, 1958.

\end{thebibliography}
\end{document}